\documentclass[hidelinks,onefignum,onetabnum]{siamart251216}
\allowdisplaybreaks

\usepackage{lipsum}
\usepackage{amsfonts}
\usepackage{graphicx}
\usepackage{epstopdf}
\usepackage{algorithmic}
\usepackage{amssymb}
\usepackage{graphicx}
\usepackage{epstopdf}
\usepackage{subcaption}
\usepackage{booktabs}
\usepackage{bm}
\usepackage{multirow}
\usepackage{mathrsfs}
\ifpdf

  \allowdisplaybreaks

  \DeclareGraphicsExtensions{.eps,.pdf,.png,.jpg}
\else
  \DeclareGraphicsExtensions{.eps}
\fi

\newsiamremark{remark}{Remark}
\newsiamremark{hypothesis}{Hypothesis}
\crefname{hypothesis}{Hypothesis}{Hypotheses}
\newsiamthm{claim}{Claim}
\newsiamremark{fact}{Fact}
\crefname{fact}{Fact}{Facts}
\newtheorem{assumption}{Assumption}[section]
\headers{Observability for high-order approximate waves}{Y. Li and X. Wang}

\title{Uniform Observability of High-Order Fully Discrete Multidimensional Wave Equations: \\Exponential Loss and Recovery\thanks{ 
 \funding{The first author was partially supported by the National Natural Science Foundation of China (No. 12301566), and by the Science and Technology Commission of Shanghai Municipality (No. 23JC1400300), and by
 the Shanghai Pujiang Program of Baiyulan Talent Plan (No. 24PJD002). } }}
\author{
Yunzhang Li\thanks{Research Institute of Intelligent Complex Systems, Fudan University, Shanghai 200433, China; Department of Mathematics, Friedrich-Alexander-Universit\"at Erlangen-N\"urnberg, 91058 Erlangen, Germany. \email { li\_yunzhang@fudan.edu.cn}.}
\and
Xiaoyang Wang\thanks{Corresponding Author. Research Institute of Intelligent Complex Systems, Fudan University, Shanghai 200433, China. \email{xiaoyangwang25@m.fudan.edu.cn}.}
}
\usepackage{amsopn}

\ifpdf
\hypersetup{
  pdftitle={Uniform Observability of High-Order Fully Discrete Multidimensional Wave Equations: Exponential Loss and Recovery},
  pdfauthor={Yunzhang Li and Xiaoyang Wang}
}
\fi

\begin{document}
\setlength{\abovedisplayskip}{0.05cm}
\setlength{\belowdisplayskip}{0.05cm}
\maketitle

\begin{abstract}
Stable discrete approximation of wave controls requires observability
estimates that remain uniform under mesh refinement. Such uniformity is known to fail for many low-order discretizations, but the quantitative understanding of observability loss for arbitrary high-order approximations is largely open, owing to their multiple spectral branches and the absence of explicit dispersion relations.
To address this gap, we establish a quantitative theory of observability loss and recovery
for fully discrete $Q^k$ local discontinuous Galerkin approximations
of the multidimensional wave equation, valid for any fixed polynomial degree $k$. Our analysis of the $(k+1)^d$-branch discrete spectrum identifies a unique physical branch that recovers the continuous propagation speed at low frequencies. Under a strict stability margin on this branch, its group velocity vanishes at high-frequency symmetry points. To quantify the resulting loss of uniform observability,
we prove that the observability constant satisfies an exponential lower bound of the form \(\exp(Ch^{-(1-\varepsilon)})\) for any
\(\varepsilon\in(0,1)\) by constructing Gevrey-localized wave packets.
We then provide a recovery strategy that restores uniform observability
for any fixed polynomial degree \(k\), with a sufficient observation time controlled by the
minimum retained physical group velocity. Finally, our numerical results suggest that higher-order approximations
retain a broader genuine physical-frequency band at a prescribed positive group-speed threshold, allowing less restrictive filtering in
wave-control computation and better preserving wave dynamics.
\end{abstract}



\section{Introduction}
\subsection{Observability of the Wave Equation}
We consider the Cauchy problem for the \(d\)-dimensional wave equation on the whole space \(\mathbb{R}^d\)
\begin{equation} \label{eq:wave_cont}
    \begin{cases}
        u_{tt}(\bm{x},t) - \Delta u(\bm{x},t) = 0, & \qquad \bm{x} \in \mathbb{R}^d, \quad t>0, \\
        u(\bm{x},0) = u^0(\bm{x}), \quad u_t(\bm{x},0) = u^1(\bm{x}), & \qquad \bm{x} \in \mathbb{R}^d,
    \end{cases}
\end{equation}
with initial data \((u^0, u^1) \in \dot{H}^1(\mathbb{R}^d) \times L^2(\mathbb{R}^d)\).
Equation \eqref{eq:wave_cont} admits waves whose spatial frequency \(\bm{\xi} \in \mathbb{R}^d\) and temporal frequency \(\tau_{\pm}\) satisfy the dispersion relation \(\tau_{\pm}(\bm{\xi})=\pm |\bm{\xi}|\), yielding a unit magnitude group velocity vector \(\nabla \tau_{\pm} = \pm\bm{\xi}/|\bm{\xi}|\). This uniform propagation is the physical mechanism underlying observability. To formalize this, we define the conserved total energy
\begin{equation} \label{eq:energy_cont}
\mathcal{E}(u(\cdot,t), u_t(\cdot,t)) := \frac{1}{2} \left( \|\nabla u(\cdot, t)\|_{L^2(\mathbb{R}^d)}^2 + \|u_t(\cdot, t)\|_{L^2(\mathbb{R}^d)}^2 \right).
\end{equation}

For the continuous problem \eqref{eq:wave_cont}, exterior observability
holds provided that the observation time is long enough for all
rays to leave the unobserved region $I$. For
\(I=[-1,1]^d\) and \(\Omega=\mathbb{R}^d\setminus I
\), the longest such travel time is \(2\sqrt{d}\). Therefore,
for every \(T>2\sqrt{d}\), there exists a constant \(C_T>0\) such that
every solution of \eqref{eq:wave_cont} satisfies
\begin{equation}\label{observability property}
\mathcal{E}(u^0, u^1) \leq C_T \int_0^T \mathcal{E}_\Omega(u(\cdot, t), u_t(\cdot, t)) \,\mathrm{d}t.
\end{equation}
Here,
\(
\mathcal{E}_\Omega(f^0, f^1) := \frac{1}{2} \left( \|\nabla f^0\|_{L^2(\Omega)}^2 + \|f^1\|_{L^2(\Omega)}^2 \right)
\)
denotes the energy on the observation region \(\Omega\).
For \(d=1\), this condition reduces to \(T>2\), as shown in \cite{MR3222005}.

This observability property \eqref{observability property} is essential in control theory and is dual to exact controllability \cite{MR4823719}. However, once system \eqref{eq:wave_cont} is discretized, the situation changes substantially. The discretization  introduces spurious high-frequency oscillations that can trap energy, thereby resulting in non-uniform observability. If the discrete observability constant blows up as $h\to0$, the corresponding control costs are no longer uniformly bounded, and convergence of the discrete controls to a control of the continuous system cannot be guaranteed. Therefore, understanding this behavior under mesh refinement is essential for reliable control computation.



\subsection{Motivation and Overview of Main Results}
\subsubsection{Motivation}
High-order approximations are attractive for wave computation because
of their high spatial accuracy and geometric flexibility, but their
observability analysis becomes substantially more involved in the fully
discrete multidimensional setting. For the \(Q^k\) piecewise polynomial
approximations considered here, the spatial discretization is governed
by a \((k+1)^d\times(k+1)^d\) matrix-valued symbol, giving rise to
\((k+1)^d\) branches whose dispersion relations are unavailable in
explicit form. This prevents the direct use of dispersion arguments
available for low-order approximations. Moreover, time discretization
modifies these branches through a nonlinear space-time dispersion
relation, making high-frequency propagation depend jointly on the
polynomial degree, spatial dimension, and time step.

These features raise four fundamental questions: whether the physical
branch develops high-frequency propagation degeneracy, how severely
this affects uniform observability as the mesh is refined, whether
uniform observability can be recovered for arbitrary fixed polynomial degree $k$,
and what advantages higher-order approximations may offer for wave-control computation. The aim of this work is to address these
questions without relying on explicit dispersion formulas.


\subsubsection{Overview of Main Results}

Our main results are summarized below,
with complete details provided in Section~\ref{Section main results}.

\begin{enumerate}
    \item[1.]\textbf{Degeneracy of Group Velocities.} Building on the spectral decomposition of the \((k+1)^d\times(k+1)^d\) matrix in Theorem~\ref{thm:spectral_structure}, we show in Theorem~\ref{thm:fully_group_velocity_limits} that the physical branch reproduces the propagation speed of the continuous wave equation at low frequencies, whereas its group velocity degenerates at high frequencies. We further identify the  conditions under which nonzero high-frequency velocities re-emerge, thereby revealing the role of the polynomial degree $k$, spatial dimension $d$, and time-step $\Delta t$ in the discrete propagation.

    \item[2.]\textbf{Exponential Observability Loss.} Theorem~\ref{non_uniform_observability} establishes an exponential lower bound \(\exp(Ch^{-(1-\varepsilon)})\) for the discrete observability constant \(C_{T,h}\), where \(h\) is the mesh size and \(\varepsilon > 0\) is arbitrarily small.
Table~\ref{tab:observability_growth} compares available lower
bounds on \(C_{T,h}\), highlighting the exponential lower bound established here.
\begin{table}[htbp]
\vspace{-1.3\baselineskip}
    \centering
    \renewcommand{\arraystretch}{1.3} 
    \caption{Known lower bounds for the discrete observability constant $C_{T,h}$ (at least).}
    \label{tab:observability_growth}
     \resizebox{0.85\textwidth}{!}{
    \begin{tabular}{l l l l l}
        \toprule
        \textbf{Year} & \textbf{Author(s)} & \textbf{Discretization} & \textbf{Scheme} & \textbf{$C_{T,h}$} \\
        \midrule
        1999 & Infante \& Zuazua \cite{MR1700042} & 1-D Semi-discrete & FD, $P^1$-FE & $\gtrsim h^{-2}$ \\
        2002 & Micu \cite{MR1912914}                  & 1-D Semi-discrete & FD            & $\gtrsim \exp(Ch^{-\frac{1}{2}})$ \\
        2014 & Marica \& Zuazua \cite{MR3222005}& 1-D Semi-discrete & $P^1$-SIPG & $\gtrsim h^{-\theta}$ \\
        \textbf{---} & \textbf{Our work} & \textbf{$d$-D Fully discrete} & \textbf{$Q^k$-LDG} & $\gtrsim \exp(Ch^{-(1-\varepsilon)})$ \\
        \bottomrule
    \end{tabular}}
\end{table}
 This exponential blow-up quantifies the loss of uniform observability
and the resulting divergence of the discrete control cost as \(h\to0\).
    \item[3.] 
    \textbf{Observability Recovery.} To overcome this exponential divergence for arbitrary fixed \(k\), Theorem~\ref{theorem restore} establishes that uniform observability can be restored via a modal-frequency filtering strategy.  Furthermore, our analysis relates the observation time to the minimum physical group velocity, while numerical evidence suggests that higher-order methods retain a larger genuine physical frequency band and  therefore allow less restrictive filtering.
\end{enumerate}




\subsection{Related Works}
\textbf{Continuous wave dynamics.} For the continuous wave equations, exact controllability and observability were established by Lions \cite{MR953547,MR963060}. The non-dispersive wave propagation yields observability constants depending only on the observation geometry and time.

\textbf{Semi-discrete approximations.}
The loss of uniform observability as $h\to0$ is well documented for semi-discrete finite difference (FD) and $P^1$-finite element (FE) approximations
\cite{MR1288099,MR1036928,MR1700042,MR2179896,MR3220862,MR3058594,MR1039237,MR1697041}.
Recovery strategies include Fourier and bi-grid filtering
\cite{MR1196839,MR2486937}, mesh modifications \cite{MR3483095}, and alternative discretizations such as mixed finite element methods
\cite{MR2207268,MR2387911,MR2654195} and tailored implicit time-stepping \cite{MR2143953}.
Micu \cite{MR1912914} established exponential blow-up of the observability constant for semi-discrete FD, while Zhang, Zheng, and Zuazua \cite{MR2449094} showed the lack of observability for time-discrete schemes with fixed time step.
Higher-order approximations introduce additional spectral complexity.
In particular, $P^2$-FE and Legendre-Galerkin formulations generate additional spurious optic branches
\cite{MR3022241,MR4191568}.

\textbf{Fully discrete approximations.} Transitioning to fully discrete approximations introduces more challenges. Negreanu and Zuazua \cite{MR2020643, MR2217389} analyzed the 
one-dimensional fully discrete FD approximation, showing that uniform observability is restored at $\Delta t=h$ via space-time dispersion cancellation. Furthermore, general transfer theorems \cite{MR2418618} can deduce fully discrete observability from conservative semi-discrete systems via resolvent estimates for appropriately filtered initial data.

\textbf{Arbitrary high-order discrete wave methods.} Arbitrary high-order piecewise polynomial approximations, originally developed for convection-diffusion systems \cite{MR1655854,MR1010597,MR983311,MR1015355,MR1619652,MR1103092}, are widely applied to wave computation. However, the multiple spurious modes inherent to these high-order discretizations complicate dispersion analysis. Marica and Zuazua \cite{MR3222005} established a polynomial blow-up of the observability constant for the $P^1$-Symmetric Interior Penalty Galerkin (SIPG) method, which is limited to second-order spatial accuracy, subsequently restoring observability via Fourier and bi-grid filtering. However, observability and filtering recovery for fully discrete
arbitrary high-order multidimensional waves remain largely unexplored.

\subsection{Outline of the Paper}
Section~\ref{Section 2} formulates the fully discrete high-order wave approximations, and Section~\ref{Section main results} states the main theoretical results, whose proofs are detailed in Section~\ref{section3}. Finally, Section~\ref{section 4} draws the conclusions.

\section{Fully Discrete High-Order Wave Approximations}\label{Section 2}
\subsection{Notation}
Let \(h\) be the mesh size on \(\mathbb R^d\), with nodes
\(\bm{x}_{\bm{j}}=(x_{j_1},\dots,x_{j_d})\), where \(x_{j_\ell+\frac12}=(j_\ell+\frac12)h\), and cells
\(\bm{I}_{\bm{j}}=\prod_{\ell=1}^d I_{j_\ell}\), \(I_{j_\ell}=(x_{j_\ell-\frac12},x_{j_\ell+\frac12})\), for \(\bm{j}\in\mathbb Z^d\).
The tensor-product polynomial space is
\[
\bm{V_{h}^k}=\bigotimes_{\ell=1}^d V_h^k = \left\{ v \in L^2(\mathbb{R}^d) : v|_{\bm{I}_{\bm{j}}} \in  Q^k(\bm{I_j}):=  P^k(I_{j_1}) \otimes\cdots\otimes P^k(I_{j_d}),\; \forall \bm{j}\in\mathbb{Z}^d \right\},
\]
where \(P^k(I)\) denotes the space of polynomials of degree at most \(k\) on an interval \(I\).

Any \(u_h \in \bm{V_{h}^k}\) is locally represented on \(\bm{I}_{\bm{j}}\) by the tensor-product monomial basis 
\(\left\{ \prod_{\ell=1}^d\left(\frac{x_{\ell} - x_{j_{\ell}}}{h/2}\right)^{l_{\ell}} \right\}_{0\le l_1,\cdots,l_d\le k}\), 
with its local coefficients collected in \(\bm{u}_{\bm{j}} \in \mathbb{C}^{(k+1)^d}\). 
The global coefficient vector 
\(\bm{U} = \{\bm{u}_{\bm{j}}\}_{\bm{j}\in\mathbb{Z}^d}\) 
is restricted to \(\ell^2(\mathbb{Z}^d; \mathbb{C}^{(k+1)^d})\). 
Define
\begin{equation*}
\begin{aligned}
\langle \bm U,\bm V\rangle
:=
\sum_{\bm j\in\mathbb Z^d}
\bm u_{\bm j}^{H}\bm v_{\bm j},\quad
\langle \bm U,\bm V\rangle_{\Omega}
:=
\sum_{\bm j\in\mathcal J_\Omega}
\bm u_{\bm j}^{H}\bm v_{\bm j},
\quad
\mathcal J_\Omega
:=
\left\{
\bm j\in\mathbb Z^d:
\bm x_{\bm j}\in\Omega
\right\}.
\end{aligned}
\end{equation*}

Fix $\lambda=\Delta t/h>0$ independently of $h$, and set
$t_n=n\Delta t$. Let \(u_h^n \in \bm{V_{h}^k}\) be the fully discrete approximation at \(t_n\) with global coefficient vector \(\bm{U}^n\).

\subsection{\texorpdfstring{The Fully Discrete Wave Approximations}{The Fully Discrete Wave Approximations}}
We construct the high-order discrete wave approximations by rewriting the wave equation \eqref{eq:wave_cont} into
\begin{equation} \label{eq:first_order_sys}
    u_{tt} - \nabla\cdot \bm{q} = 0, \qquad \bm{q} - \nabla u = 0,
\end{equation}
where \(\bm{q}=(q_1,\cdots,q_d)\) is the auxiliary vector variable.
Multiplying the two equations in \eqref{eq:first_order_sys} by
\(w\in \bm{V_h^k}\) and
\(\bm v\in(\bm V_h^k)^d\), respectively, and integrating by parts
over \(\bm I_{\bm j}\), we obtain the local weak formulation
\begin{align}
    \int_{\bm I_{\bm j}} (u_h)_{tt} w \,\mathrm d\bm x
    &= \sum_{\ell=1}^d \biggl[
        \int_{x_{j_\ell+\frac12}}\!\! \widetilde{q}_{\ell,h} \, w^- \,\mathrm dS_\ell
        - \int_{x_{j_\ell-\frac12}}\!\! \widetilde{q}_{\ell,h} \, w^+ \,\mathrm dS_\ell
    \biggr]-\int_{\bm I_{\bm j}}\! \bm q_h \cdot \nabla w \,\mathrm d\bm x
     , \label{eq:weak_b} \\
    \int_{\bm I_{\bm j}} \bm q_h \cdot \bm v \,\mathrm d\bm x
    &= \sum_{\ell=1}^d \biggl[
        \int_{x_{j_\ell+\frac12}}\!\! \widetilde u_h \, v_\ell^- \,\mathrm dS_\ell
        - \int_{x_{j_\ell-\frac12}}\!\! \widetilde u_h \, v_\ell^+ \,\mathrm dS_\ell
    \biggr]-\int_{\bm I_{\bm j}} u_h (\nabla\cdot \bm v) \,\mathrm d\bm x
    , \label{eq:weak_a}
\end{align}
where $\widetilde{u_h} = u_h^-$ and $\widetilde{q}_{\ell,h} = q_{\ell,h}^+$ denote the alternating fluxes. Substituting equation \eqref{eq:weak_a} into \eqref{eq:weak_b} and applying a central difference time discretization yields
\begin{equation} \label{eq:local_fully_discrete}
    \bm{M}^h \frac{\bm{u}_{\bm{j}}^{n+1} - 2\bm{u}_{\bm{j}}^n + \bm{u}_{\bm{j}}^{n-1}}{(\Delta t)^2}
    + \sum_{\bm{\alpha}\in\{-1,0,1\}^d} \bm{K}_{\bm{\alpha}}^h \bm{u}_{\bm{j}+\bm{\alpha}}^n = \bm{0},
\end{equation}
where \(\bm{M}^h\) and \(\bm{K}_{\bm{\alpha}}^h\) are the \((k+1)^d\times(k+1)^d\) local matrices. Thus, although the summation is written over
\(\{-1,0,1\}^d\), only the central block and the \(2d\) coordinate-neighbor blocks
are nonzero. These matrices satisfy the symmetry relations
\(
(\bm{K}_{\bm{\alpha}}^h)^T = \bm{K}_{-\bm{\alpha}}^h, \, (\bm{M}^h)^T = \bm{M}^h,
\)
which will be used in the subsequent analysis.

Aggregating the local system \eqref{eq:local_fully_discrete} over all \(\bm{j}\in\mathbb{Z}^d\), the global fully discrete approximations can be written in the following compact form
\begin{equation} \label{eq:global_fully_discrete}
    \mathbb{M}^h \frac{\bm{U}^{n+1} - 2\bm{U}^n + \bm{U}^{n-1}}{(\Delta t)^2} + \mathbb{K}^h \bm{U}^n = \bm{0}, \qquad n \ge 1.
\end{equation}
Here, the matrix \(\mathbb{M}^h\) is block-diagonal with diagonal blocks \(\bm{M}^h\), while the matrix \(\mathbb{K}^h\) is a block Toeplitz operator generated by the finite family \(\{\bm{K}_{\bm{\alpha}}^h\}_{\bm{\alpha}\in\{-1,0,1\}^d}\).
\begin{remark}
For \(d=1\), the high-order discrete wave approximations become:
\begin{itemize}
    \item The matrix \(\bm{M}^h\in\mathbb{R}^{(k+1)^d\times(k+1)^d}\) reduces to matrix \(\bm{M}_1^h\in\mathbb{R}^{(k+1)\times(k+1)}\) and satisfies \(\bm{M}^h = (\bm{M}_1^h)^{\otimes d}\), with \((\bm{M}_1^h)^{\otimes d}:=\underbrace{\bm{M}_1^h \otimes \cdots \otimes \bm{M}_1^h}_{d\text{ times}}\);
    \item The stiffness blocks \(\bm{K}_{\bm{\alpha}}^h\) for \(\bm{\alpha}\in\{-1,0,1\}^d\) reduce to \(\bm{K}_{\alpha}^h\), with \(\alpha\in\{-1,0,1\}\).
    \end{itemize}
The local fully discrete system \eqref{eq:local_fully_discrete} then becomes
\begin{equation}
    \bm{M}_1^h \frac{\bm{u}_j^{n+1} - 2\bm{u}_j^n + \bm{u}_j^{n-1}}{(\Delta t)^2}
    + \bm{K}^h_{-1}\bm{u}_{j-1}^n + \bm{K}^h_0\bm{u}_j^n + \bm{K}^h_{+1}\bm{u}_{j+1}^n = \bm{0}.
\end{equation}
\end{remark}
\subsection{Discrete energy}
For $\ell=1,\ldots,d$, let
$\bm Q_\ell=\{\bm q_{\ell,\bm j}\}_{\bm j\in\mathbb Z^d}$,
where $\bm q_{\ell,\bm j}$ is the coefficient vector of the
auxiliary variable $q_{\ell,h}$ on $\bm I_{\bm j}$ of equation \eqref{eq:weak_a}.
We define the total discrete energy $\mathcal E_h^n$,
conserved by \eqref{eq:global_fully_discrete},
and the observation energy $\mathcal E_{h,\Omega}^n$ by
\begin{equation} \label{eq:discrete_energy}
    \mathcal{E}_h^{n} := \frac{1}{2} \langle \mathbb{M}^h\frac{\bm{U}^{n+1} - \bm{U}^n}{\Delta t}, \frac{\bm{U}^{n+1} - \bm{U}^n}{\Delta t}\rangle + \frac{1}{2}\mathrm{Re}\langle \mathbb{K}^h\bm U^{n+1},\bm U^n  \rangle,
\end{equation}
\begin{equation}\label{eq:observe_energy}
    \begin{aligned}
\mathcal E_{h,\Omega}^{n}
:={}&
\frac12\left\langle
\mathbb M^h
\frac{\bm U^{n+1}-\bm U^n}{\Delta t},
\frac{\bm U^{n+1}-\bm U^n}{\Delta t}
\right\rangle_{\Omega}
+\frac14\sum_{m=n}^{n+1}\sum_{\ell=1}^d
\left\langle
\mathbb M^h\bm Q_\ell^m,
\bm Q_\ell^m
\right\rangle_{\Omega}.
\end{aligned}
\end{equation}

\subsection{Frequency-Domain Formulation}
We analyze the dispersion of the fully discrete system \eqref{eq:global_fully_discrete} via the $d$-dimensional semi-discrete Fourier transform (SDFT). For \(\bm{F} = \{\bm{f}_{\bm{j}}\}_{\bm{j} \in \mathbb{Z}^d} \in \ell^2(\mathbb{Z}^d; \mathbb{C}^{(k+1)^d})\), its SDFT \(\widehat{\bm{F}}(\bm{\xi})\) over \(\Pi_h^d = [-\frac{\pi}{h}, \frac{\pi}{h}]^d\) and the inverse transform are given by
\begin{equation}\label{fourier transform}
    \widehat{\bm{F}}(\bm{\xi})\! = \!h^d \!\sum_{\bm{j} \in \mathbb{Z}^d} \bm{f}_{\bm{j}} e^{-\sqrt{-1} \bm{\xi} \cdot \bm{x}_{\bm{j}}},\, 
    \bm{f}_{\bm{j}} \!=\! \frac{1}{(2\pi)^d} \!\int_{\Pi_h^d} \widehat{\bm{F}}(\bm{\xi}) e^{\sqrt{-1} \bm{\xi} \cdot \bm{x}_{\bm{j}}} \,\mathrm{d}\bm{\xi},\,\bm{\xi}=(\xi_1,\cdots,\xi_d).
\end{equation}

Application of the SDFT to the system \eqref{eq:local_fully_discrete} yields a family of decoupled discrete-time systems in the frequency domain. For any \(\bm{\xi} \in \Pi_h^d\), \(\widehat{\bm{U}}^n(\bm{\xi})\) satisfies
\begin{equation}\label{fully Fourier mode}
\bm{M}^h \bigl(\widehat{\bm{U}}^{n+1} + \widehat{\bm{U}}^{n-1} - 2\widehat{\bm{U}}^n\bigr)(\bm{\xi}) + (\Delta t)^2 \mathcal{K}^h(\bm{\xi}) \widehat{\bm{U}}^n(\bm{\xi}) = \bm{0},
\end{equation}
where 
\(
\mathcal{K}^h(\bm{\xi})=\sum_{\ell=1}^d 
(\bm{M}^h_1)^{\otimes(\ell-1)}\otimes \mathcal{K}_1^h(\xi_{\ell})\otimes (\bm{M}^h_1)^{\otimes(d-\ell)}
\)
is the \((k+1)^d\times(k+1)^d\) symbol matrix, with
\(
\mathcal{K}_1^h(\xi_\ell)=\bm{K}^h_0+\bm{K}^h_{-1}e^{-\sqrt{-1}\xi_\ell h}+\bm{K}^h_{+1}e^{\sqrt{-1}\xi_\ell h},\,\ell=1,\dots,d,
\)
denoting the one-dimensional symbol matrix in each coordinate direction.

Thus, the analysis of the fully discrete approximations \eqref{eq:global_fully_discrete} reduces to the study of a family of \((k+1)^d \times (k+1)^d\) generalized eigenvalue problems:
\begin{equation}\label{fully eigenvalue problem}
\left(\mathcal K^h(\boldsymbol{\xi}) - \sigma(\boldsymbol{\xi})\, \bm{M}^h\right) \bm{v}(\boldsymbol{\xi}) = \mathbf 0, \qquad 
\bm{v}(\boldsymbol{\xi}) \in \mathbb{C}^{(k+1)^d} \setminus \{\mathbf 0\},
\end{equation}
where 
\(
\sigma(\boldsymbol{\xi}) = \frac{4}{(\Delta t)^2} \sin^2\!\left(\frac{\omega(\boldsymbol{\xi}) \Delta t}{2}\right),\) with the stability condition
\( 
\max_{\boldsymbol{\xi}\in\Pi_h^d} \sigma(\boldsymbol{\xi}) \le \frac{4}{(\Delta t)^2}.
\)
\section{Main Results}\label{Section main results}
In this section, we present the detailed main results of observability for the fully discrete
arbitrary high-order wave approximations. All the results of this paper can be easily extended to $\Omega=\mathbb{R}^d\setminus\prod_{i=1}^d[a_i,b_i]$.
\subsection{Spectral Structure}
\textbf{Our first main result} establishes the algebraic structure of the eigensystem \eqref{fully eigenvalue problem}.
\begin{theorem}\label{thm:spectral_structure}
The eigenvalue problem \eqref{fully eigenvalue problem} admits $(k+1)^d$ real-analytic eigenpairs 
\(\{\sigma_{m_1,\cdots,m_d}(\bm\xi), \bm{v}_{m_1,\cdots,m_d}(\bm \xi)\}_{m_1,\cdots,m_d=1}^{k+1}\) on an open neighborhood of $\Pi_h^d$ in $\mathbb R^d$,
with an $\bm M^h$-orthonormal eigenbasis of $\mathbb C^{(k+1)^d}$
and the following properties.
\begin{itemize}
    \item[{\tiny $\blacksquare$}] 
    There is a unique principal branch indexed by \(\bm{m} =(ph,\dots,ph)\) such that
    \begin{equation*}
        \sigma_{\bm m}(\boldsymbol 0)=0,\qquad 
        \bm{v}_{\bm m}(\boldsymbol 0)\propto \bm{e}_0^{\otimes d},\qquad 
        \nabla^2\sigma_{\bm m}(\boldsymbol 0)=2\bm{I}_d,\qquad\sigma_{\bm{m} }(\boldsymbol{\xi}) = \sigma_{\bm{m} }(-\boldsymbol{\xi}).
    \end{equation*}
    \item[{\tiny $\blacksquare$}] 
    For \(\bm{m}\ne(ph,\dots,ph)\), the corresponding eigenvalue satisfies \(\sigma_{\bm{m} }(\boldsymbol 0)>0.\)
\end{itemize}
\end{theorem}
Theorem~\ref{thm:spectral_structure} guarantees that all eigenvalues are real and nonnegative. Under the condition
\(\max_{\bm \xi}\sigma_{\bm m}(\boldsymbol{\xi})
\leq \frac{4}{(\Delta t)^2},
\)
they determine real discrete temporal frequencies. This spectral separation distinguishes the physical mode from the
spurious modes. 
\begin{definition}\label{definition of fully physical mode}
The discrete temporal frequency is defined as
\[
\omega_{\bm{m}}(\boldsymbol{\xi}) := \frac{2}{\Delta t}\arcsin{\left(\frac{\sqrt{\sigma_{\bm{m}}(\boldsymbol{\xi})}\Delta t}{2}\right)},\qquad 
\bm{m}\in\{1,\dots,k+1\}^d,\qquad \boldsymbol{\xi}\in\Pi_h^d.
\]
The pair $(\omega_{\bm{ph}}(\boldsymbol{\xi}), \bm{v}_{\bm{ph}}(\boldsymbol{\xi}))$ is the physical mode, where $\bm{ph}=(ph,\dots,ph)$, while the remaining $(k+1)^d-1$ pairs are spurious modes.
\end{definition}
\begin{remark}\label{rem:1d_to_nd}
Owing to the tensor-product structure, the \((k+1)^d\) eigenvalues of the \(d\)-dimensional system admit the additive decomposition
\(
\sigma_{m_1,\dots,m_d}(\bm{\xi}) = \sum_{i=1}^d \sigma^{1}_{m_i}(\xi_i),
\)
while the corresponding eigenvectors decompose as
\(
\bm v_{\bm m}(\bm{\xi}) = \bigotimes_{\ell=1}^d \bm v_{m_\ell}(\xi_\ell).
\)
Here, for each \(i=1,\dots,d\), the pair \((\sigma^{1}_{m_i}(\xi_i), \bm v_{m_i}(\xi_i))\), with \(m_i\in\{1,\dots,k+1\}\), is an eigenpair of the one-dimensional generalized eigenvalue problem
\begin{equation}\label{1d eigenvalues}
    \left(\mathcal K^h_1(\xi_i) - \sigma^{1}(\xi_i)\, \bm{M}^h_1\right) \bm v(\xi_i) = \bm 0, \qquad 
    \bm v(\xi_i) \in \mathbb{C}^{k+1} \setminus \{\bm 0\}.
\end{equation}
\end{remark}
\subsection{Degeneracy of the Physical Group Velocity}
\textbf{Our second main result} characterizes the asymptotic behavior of physical group velocities at the long-wave limit and the highest-frequency limit of $\Pi_h^d$.

\begin{theorem}\label{thm:fully_group_velocity_limits}
Assume the stability condition
\(
(\Delta t)^2
\max_{\boldsymbol{\xi}\in\Pi_h^d}
\sigma_{\bm{ph}}(\boldsymbol{\xi})
\le 4.
\)
Then the physical group velocity
\(\bm V_g^{\bm{ph}}(\bm{\xi})
:=\nabla\omega_{\bm{ph}}(\bm{\xi})\)
satisfies
\begin{equation}\label{physical velocity}
\lim_{f\to0^+}\bm V_g^{\bm{ph}}(f\bm w)
=\bm w,\qquad |\bm w|=1.
\end{equation}
Let
\(
\boldsymbol{\xi}_*=\frac{\pi}{h}\bm s,
\,
\bm s\in\{-1,1\}^d.
\)
Then, for every \(\bm a\in[0,\infty)^d\) with \(|\bm a|=1\),
\begin{equation}
\lim_{f\to0^+}
\bm V_g^{\bm{ph}}
\left(
\boldsymbol{\xi}_*
-f\,\mathrm{diag}(\boldsymbol{s})\boldsymbol{a}
\right)
=
\begin{cases}
\mathrm{diag}(\boldsymbol{s})\boldsymbol{a},
& \text{if }
(k,\Delta t)=\left(0,\dfrac{h}{\sqrt d}\right),\\[1ex]
\mathbf 0,
& \text{if }
(\Delta t)^2 \max\limits_{\boldsymbol{\xi}\in\Pi_h^d} \sigma_{\bm{ph}}(\boldsymbol{\xi}) < 4.
\end{cases}
\end{equation}
\end{theorem}

In view of Theorem~\ref{thm:fully_group_velocity_limits}, we clarify the propagation failure of high-frequency discrete wave packets and demonstrate the departure from the uniform propagation of continuous waves, revealing the root cause of observability loss.
\begin{remark}
The same argument applies to the $(k+1)^d-1$ spurious branches. When the strict stability condition
\(
4-(\Delta t)^2\sigma_{\bm m}(\boldsymbol\xi)>0
\)
holds, vanishing group velocities occur not only for the physical
mode, but also for the spurious modes at these symmetry points. For brevity, we omit the details for spurious modes.
\end{remark}

\subsection{Exponential Observability Loss}
For $\Omega=\mathbb R^d\setminus[-1,1]^d$ and $T>0$,
the approximation is uniformly observable if the following
inequality holds for every approximate solution with constants
$C_{T,h}>0$ uniformly bounded as $h\to0$.
\begin{equation} \label{eq:discrete_observability}
    \mathcal{E}_h^{0} \le C_{T,h} \Delta t \sum_{n=0}^{N-1} \mathcal{E}_{h,\Omega}^{n}.
\end{equation}

\textbf{Our third main result} gives a quantitative characterization of the exponential loss of uniform observability induced by the group-velocity degeneracy of Theorem~\ref{thm:fully_group_velocity_limits}. To achieve this, we construct a high-frequency Gevrey-localized wave packet. Consider the discrete initial data pair $(\widehat{\bm{U}}^0(\bm\xi), \widehat{\bm{U}}^1(\bm\xi))$ defined in $\Pi_h^d$ as follows
\begin{align}\label{eq:initial_data_construction}
    \widehat{\bm{U}}^0(\boldsymbol{\xi}) = r(\boldsymbol{\xi}) h^{\frac{(1+\gamma) d}{2}} \chi_{\rho}(\boldsymbol{\xi}) e^{-\sqrt{-1} \boldsymbol{\xi}\cdot \bm{x}_c} \bm{v}_{\bm{ph}}(\boldsymbol{\xi}), \quad
    \widehat{\bm{U}}^1(\boldsymbol{\xi}) = e^{-\sqrt{-1} \omega_{\bm{ph}}(\boldsymbol{\xi}) \Delta t} \widehat{\bm{U}}^0(\boldsymbol{\xi}), 
\end{align}
where
\begin{itemize}
    \item[{\tiny $\blacksquare$}] $r(\boldsymbol{\xi}) = \left( \sigma_{\bm{ph}}(\boldsymbol{\xi}) \langle \bm{M}^h \bm{v}_{\bm{ph}}(\boldsymbol{\xi}), \bm{v}_{\bm{ph}}(\boldsymbol{\xi}) \rangle \right)^{-\frac{1}{2}}$;
   \item[{\tiny $\blacksquare$}] $\bm{x}_c \in (-1, 1)^d$, $\gamma \in (0, 1)$ fixed but arbitrary;
   \item[{\tiny $\blacksquare$}] $\chi_\rho(\boldsymbol{\xi}) = \prod_{\ell=1}^d \psi\left(\frac{\xi_{\ell} - \xi_{c\ell}}{\rho}\right)$, with $\psi \in C_c^\infty([-1,1])\cap G^s$, $\psi\not\equiv0$, $s>1$, centered at $\boldsymbol{\xi}_c=(\xi_{c1},\cdots,\xi_{cd})$ near $(\pi/h,\cdots,\pi/h)$ with $\rho = h^{-\gamma}$.
\end{itemize}
By localizing the initial data around the highest frequency via a Gevrey cut-off function and projecting it onto the physical mode eigenvector, we exclude the interference from spurious modes. The factor
\(e^{-\sqrt{-1}\omega_{\bm{ph}}(\boldsymbol\xi)\Delta t}\)
selects the temporal branch
\(e^{-\sqrt{-1}\omega_{\bm{ph}}(\boldsymbol\xi)t_n}\). Governed by the group-velocity degeneracy, on every fixed time interval,
the energy reaching the observation region is exponentially small as
\(h\to0\).

To quantify the observability loss, we formulate the following Assumption.
\begin{assumption} \label{assump:observability}
   There exists a constant \(\epsilon>0\), independent of \(h\), such that
\begin{enumerate}
    \renewcommand{\labelenumi}{\textbf{(A\arabic{enumi})}}
    \item The strict stability condition holds:
    \(
        (\Delta t)^2
\max_{\boldsymbol\xi\in\Pi_h^d}
\sigma_{\bm{ph}}(\boldsymbol\xi)
\le 4-\epsilon.
    \)
    \item The  stability condition holds:
    \(
        (\Delta t)^2\max_{\bm\xi}\sigma_{\bm m}(\bm\xi) \le 4,\,\bm m\ne \bm{ph}.
    \)
    \item For some fixed \(\kappa>1\),
    \(
        \boldsymbol\xi_c
=
\left(\frac{\pi}{h}-\kappa h^{-\gamma}\right)\boldsymbol1,
    \) and \(h\) is sufficiently small that
\(\prod_{\ell=1}^d [\xi_{c\ell} - \rho, \xi_{c\ell} + \rho]\subset\Pi_h^d\).
\end{enumerate}
\end{assumption}

The following Theorem establishes the exponential blow-up of the observability constant under Assumption~\ref{assump:observability} via a quantitative non-stationary phase analysis. 
\begin{theorem}\label{non_uniform_observability}
Under Assumption~\ref{assump:observability}, for every fixed \(T>0\) and every \(\varepsilon\in(0,1)\), uniform observability fails for the approximations \eqref{eq:global_fully_discrete}, with $C_{T,h}$ growing exponentially as $h \to 0$
\begin{equation*}
    C_{T,h} \geq C_{1,3}(k,T,\lambda, \varepsilon, d) \exp\left(\frac{C_4(k,T,\lambda,\varepsilon,d)}{h^{1-\varepsilon}}\right),
\end{equation*}
where $C_{1,3}(k,T, \lambda,\varepsilon,d)>0$ and $C_4(k,T,\lambda,\varepsilon,d)>0$ are independent of $h$.
\end{theorem}
\begin{remark}\label{rem:critical_case}
Assumption~\ref{assump:observability} excludes the critical case
$\max_{\bm\xi}\sigma_{\bm{ph}}(\bm\xi) = \frac{4}{(\Delta t)^2}$ ($k=0, \sqrt{d}\Delta t=h$). For $d\ge2$, even at this critical
condition, the group velocity still vanishes at mixed symmetry
points such as $(\frac{\pi}{h},0,\ldots,0)$.
\end{remark}

\subsection{Filtering Recovery}
By Theorem~\ref{thm:fully_group_velocity_limits} and the low-frequency expansion of the physical dispersion relation, the physical group velocity remains uniformly separated from zero in a neighborhood of the origin, as stated in the following Lemma.
\begin{lemma}\label{lem:local_positivity}
    For the approximations \eqref{eq:global_fully_discrete}, there exist constants $\eta_k\in (0,\pi)$ and $c_k>0$, independent of $h$, such that the physical group velocity satisfies
    \[
        |\bm{V}_g^{\bm{ph}}(\bm\xi)| \ge c_k,\qquad
    \bm\xi\in
    \left[-\frac{\eta_k}{h},\frac{\eta_k}{h}\right]^d\setminus\{\bm 0\}.
    \]
\end{lemma}
Motivated by Lemma~\ref{lem:local_positivity}, we set $\delta_k = 1 - \frac{\eta_k}{\pi}$. The band $I^d_{\delta_k} = [\frac{-(1-\delta_k)\pi}{h}, \frac{(1-\delta_k)\pi}{h}]^d$, with the origin removed, then lies within the region where $|\bm{V}_g^{\bm{ph}}(\bm\xi)| \ne 0$. Consequently, there exists a constant $v_g(\delta_k)
:=
\inf_{\bm\xi\in I_{\delta_k}^d\setminus\{\bm0\}}
\left|\bm V_g^{\bm{ph}}(\bm\xi)\right|> 0$ such that $|\bm{V}_g^{\bm{ph}}(\bm\xi) |\ge v_g(\delta_k)$ for all $\bm\xi \in I_{\delta_k}^d\setminus\{\bm0\}$.

\textbf{Our final main result} gives the recovery strategies for the observability loss in Theorem~\ref{non_uniform_observability}. The following Theorem achieves this by projecting the initial data onto the physical branch and restricting it to frequencies with positive $|\bm{V}_g^{\bm{ph}}|$.
\begin{theorem}\label{theorem restore}
Fix \(h_0>0\). Under Assumption~\ref{assump:observability}(A1), for $\delta_k = 1 -\frac{\eta_k}{\pi} \in (0,1)$, where $\eta_k$ is given in Lemma~\ref{lem:local_positivity}, there exists a positive constant $T^*(\delta_k)$ such that for any $T > T^*(\delta_k)$, the approximations \eqref{eq:global_fully_discrete} with initial data restricted to the filtered space $\mathcal{V}_h^{\delta_k}\times\mathcal{V}_h^{\delta_k}$ are uniformly observable
    \[\mathcal{E}_h^{0} \le \frac{1}{T-T^*(\delta_k)} \Delta t \sum_{n=0}^{N-1} \mathcal{E}_{h,\Omega}^{n},\]
    where \(0<h\le h_0\), \(N\in\mathbb N^+\),
and \(\Delta t=\lambda h=T/N\). Here, the filtered space $\mathcal{V}_h^{\delta_k}$ is defined as 
\begin{equation*}
    \mathcal{V}_h^{\delta_k} = \left\{ \bm{U} \in \ell^2(\mathbb{Z}^d; \mathbb{C}^{(k+1)^d}) \ \Big| \ \,\operatorname{supp}(\widehat{\bm{U}}(\bm\xi)) \subset I^d_{\delta_k}, \ \widehat{\bm{U}}(\bm\xi) \in \text{span}\{\bm{v}_{\bm {ph}}(\bm\xi)\} \right\}.
\end{equation*}
Moreover, $T^*(\delta_k)=C/v_g(\delta_k)$, where the positive constant $C$ is independent of $h$.
\end{theorem}


To illustrate how the polynomial degree \(k\) affects high-frequency propagation, we compute the physical group velocity \(\bm V_g^{\boldsymbol{ph}}\) in two dimensions for the fully discrete \(Q^k\) approximations in Figure~\ref{fig:fully_group_velocity_2d}, using \(h=1\) and \(k=0,1,2\). For each \(k\), we set the CFL ratio \(\lambda=\Delta t/h\) to \(80\%\) of its stability limit. The plots suggest that higher degrees may better preserve physical wave speeds and dynamics over a wider frequency band.
\begin{figure}[h!]\vspace{-0.5cm}
		\centering
		\includegraphics[width=\linewidth]{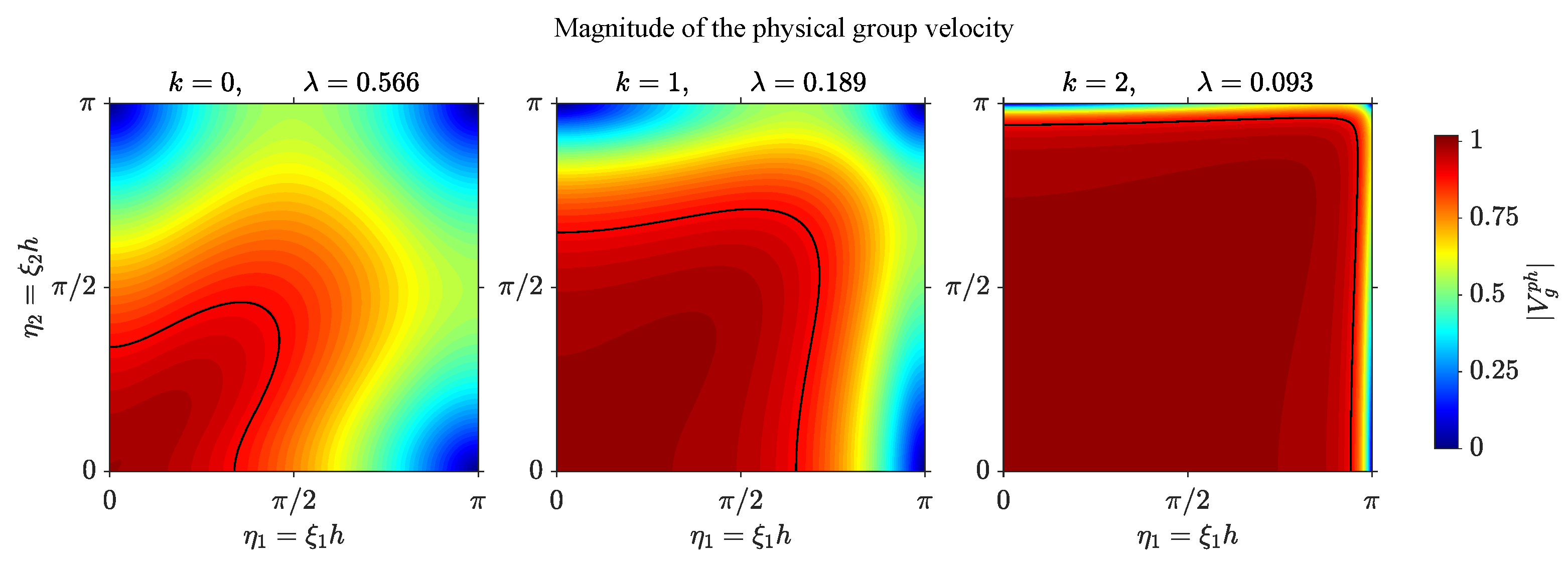}

	\caption{Magnitude of the physical group velocity $\bm{V}_g^{\bm{ph}}$ for \(Q^k\) approximations with \(k=0,1,2\). The black curves denote the \(0.9\) level sets of \(\left|\bm{V}_g^{\bm{ph}}\right|\).}
	\label{fig:fully_group_velocity_2d}
	\vspace{-0.7cm}
\end{figure}

At the common CFL ratio
\(\lambda=0.8\min_{k=0,1,2,3}\lambda_{k,\max}\),
Figure~\ref{fig:fully_group_velocity_2d1} compares the minimum
physical group speed over \([-\eta_k,\eta_k]^2\) and the largest
cutoff \(\eta_k\) satisfying a prescribed minimum-speed threshold,
showing that higher-order approximations admit wider retained
bands and larger minimum speeds at a fixed cutoff.

\begin{figure}[h!]\vspace{-0.04cm}
		\centering
		\includegraphics[width=0.8\linewidth]{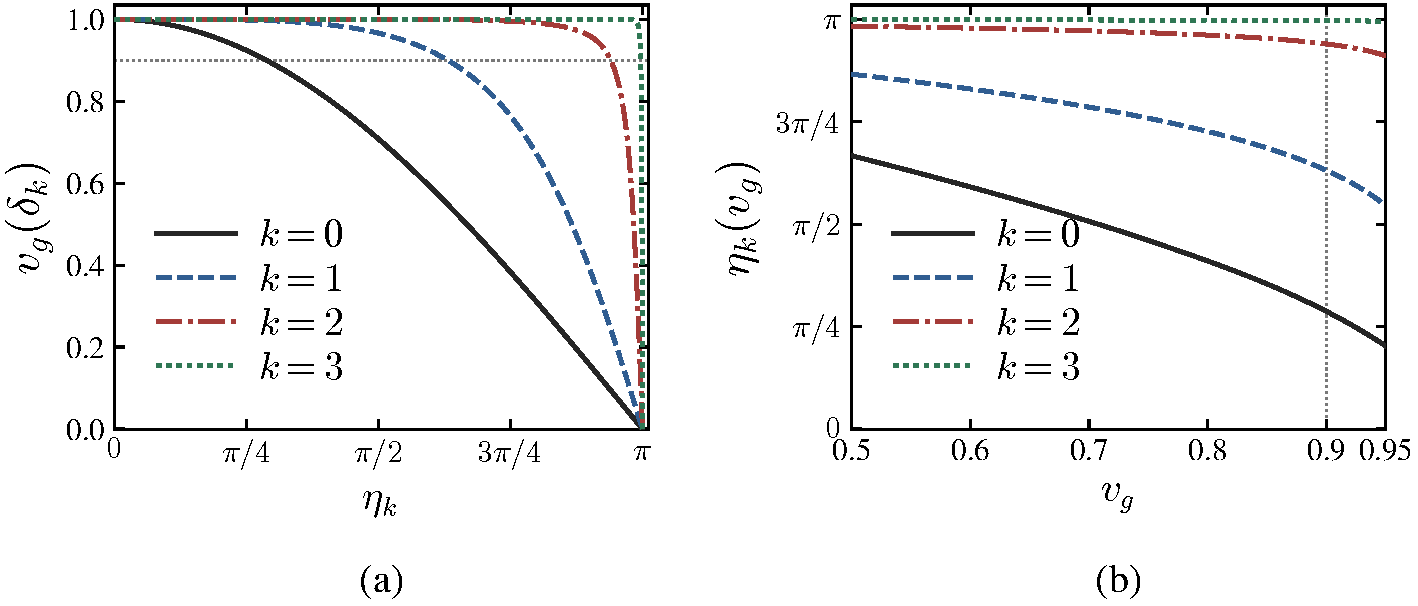}

	\caption{Minimum physical group speed \(v_g(\delta_k)\) (a) and
admissible cutoff \(\eta_k(v_g)\) (b) for \(k=0,1,2, 3\)
with \(\lambda=0.054\).}
	\label{fig:fully_group_velocity_2d1}
	\vspace{-0.9cm}
\end{figure}
\begin{remark}\label{rem:higher_order_advantage} While the bound $\eta_k$ ensures that  $|\bm{V}_g^{\bm{ph}}|$ remains positive within $I_{\delta_k}^d\setminus\{\bm0\}$,  Figures~\ref{fig:fully_group_velocity_2d}--\ref{fig:fully_group_velocity_2d1} indicate that higher degrees can retain wider
physical-frequency bands while maintaining the same prescribed
positive lower bound on the group speed. This provides numerical
evidence for less restrictive frequency filtering at a fixed speed threshold. Although an analytic proof of this monotonicity is not pursued here, this observation suggests that higher-order methods retain a broader genuine physical frequency band, thereby  better preserving the original continuous wave dynamics.
\end{remark}

\section{Proof of Main Results}\label{section3}
\subsection{Proof of Theorem~\ref{thm:spectral_structure}}\label{subsection4.1}
For the sake of clarity, the proof of Theorem~\ref{thm:spectral_structure} is divided into the following sequence of Lemmas.
\begin{lemma}\label{lemma3.1}
    The symbol matrix \(\mathcal K^h(\boldsymbol{\xi})\) is Hermitian positive semi-definite. Moreover, at every corner \(\boldsymbol{\xi}_*\in\{\pm \frac{\pi}{h}\}^d\), \(\mathcal K^h(\boldsymbol{\xi}_*)\) is Hermitian positive definite.
\end{lemma}
\begin{proof}
By the Kronecker-sum structure
\begin{equation}\label{Kronecker-sum structure}
    \mathcal K^h(\boldsymbol{\xi})=\sum_{\ell=1}^d (\bm{M}_1^h) ^{\otimes(\ell-1)}\otimes \mathcal{K}_1^h(\xi_\ell)\otimes (\bm{M}_1^h)^{\otimes(d-\ell)},
\end{equation}
the \(d\)-dimensional properties follow from the one-dimensional ones,
since Kronecker products with the positive definite mass matrix
\(\bm M_1^h\) preserve these properties.

By the properties of the matrices $K^h_{\pm1}$, $K^h_0$ and $M^h_1$, a direct calculation verifies that $\mathcal{K}_1^h(\xi)$ is Hermitian.
Since $\mathbb{K}^h$ is positive semi-definite, for any sequence $\bm{U} = \{\bm{u}_j\}_{j \in \mathbb{Z}} \in \ell^2(\mathbb{Z}; \mathbb{C}^{k+1})$, the global energy is non-negative
    \begin{equation}\label{eq:global_energy_nonneg}
        \|\bm{U}\|_{\mathbb{K}^h}^2 =  \sum_{j \in \mathbb{Z}} \bm{u}_j^H \left( \bm{K}_{-1}^h \bm{u}_{j-1} + \bm{K}_0^h \bm{u}_j + \bm{K}_{+1}^h \bm{u}_{j+1} \right) \geq 0.
    \end{equation}

    Applying Parseval's identity via the SDFT transforms the spatial summation \eqref{eq:global_energy_nonneg} into the frequency domain
    $[-\frac{\pi}{h}, \frac{\pi}{h}]$,
    \(
        \|\bm{U}\|_{\mathbb{K}^h}^2 = \frac{1}{2\pi h} \int_{-\frac{\pi}{h}}^{\frac{\pi}{h}} \widehat{\bm{u}}(\xi)^H \mathcal{K}_1^h(\xi) \widehat{\bm{u}}(\xi) \mathrm{d}\xi \geq 0.
\)

    Suppose, by contradiction, there exists $\xi_0 \in [-\frac{\pi}{h}, \frac{\pi}{h}]$ and a non-zero vector $\bm{v}_0 \in \mathbb{C}^{k+1}$ such that
    $\bm{v}_0^H \mathcal{K}_1^h(\xi_0) \bm{v}_0 = -c < 0$.
    By the continuity of $\mathcal{K}_1^h(\xi)$, there exists a local neighborhood
$\mathcal{N}_\delta(\xi_0) = (\xi_0 - \delta, \xi_0 + \delta)\cap
\left[-\frac{\pi}{h},\frac{\pi}{h}\right]$ such that 
\(\bm{v}_0^H \mathcal{K}_1^h(\xi) \bm{v}_0 \leq -\frac{c}{2} < 0, \, \forall \xi \in \mathcal{N}_\delta(\xi_0).\)
We construct the sequence $\bm{U}$ by defining its Fourier transform as $\widehat{\bm{u}}(\xi) = \bm{v}_0 \chi_{\mathcal{N}_\delta(\xi_0)}(\xi)$, where $\chi$ denotes the indicator function. So we have
\[\|\bm{U}\|_{\mathbb{K}^h}^2 = \frac{1}{2\pi h} \int_{\mathcal{N}_\delta(\xi_0)}\bm{v}_0^H \mathcal{K}_1^h(\xi) \bm{v}_0 \mathrm{d}\xi \leq \frac{1}{2\pi h} \left(-\frac{c}{2}\right)|\mathcal{N}_\delta(\xi_0)| < 0.\]
Therefore, $\mathcal{K}_1^h(\xi)$ must be positive semi-definite for all $\xi \in [-\frac{\pi}{h}, \frac{\pi}{h}]$.

To prove the positive definiteness of $\mathcal{K}_1^h(\frac{\pi}{h})$, we proceed by contradiction.  Suppose there exists a non-zero vector $\bm{v} \in \mathbb{C}^{k+1}$ such that $\bm{v}^H \mathcal{K}_1^h(\frac{\pi}{h}) \bm{v} = 0$. For any integer $N > 0$, we define the sequence $\bm{U}^N = \{\bm{u}^N_j\}_{j \in \mathbb{Z}} \in \ell^2(\mathbb{Z}; \mathbb{C}^{k+1})$ by
\(
    \bm{u}^N_j =
    \bm{v}(-1)^j\mathbf{1}_{\{|j|\le N\}} .
\)
Substituting $\bm{U}^N$ into  \eqref{eq:global_energy_nonneg}, we observe that for any $|j| \le N-1$, the local evaluation matches $\mathcal{K}_1^h(\frac{\pi}{h})$, which vanishes by our assumption
    \begin{equation*}
        \bm{v}^H (-1)^j \left[ \bm{K}_{-1}^h \bm{v} (-1)^{j-1} + \bm{K}_0^h \bm{v} (-1)^j + \bm{K}_{+1}^h \bm{v} (-1)^{j+1} \right] 
        = \bm{v}^H \mathcal{K}_1^h(\frac{\pi}{h}) \bm{v} = 0.
    \end{equation*}
For \(|j|\le N-1\), the interior contributions vanish.
All remaining terms involve only the indices \(j\in\{-N-1,-N,N,N+1\}.\) 
Since the number of these terms is independent of \(N\), and each of
them is bounded in terms of \(\bm v\) and the fixed matrices
\(\bm K_{-1}^h,\bm K_0^h,\bm K_{+1}^h\), there exists a constant
\(C_h(\bm v)>0\), independent of \(N\), such that $\|\bm{U}^N\|_{\mathbb{K}^h}^2 \le C_h(\bm{v})$ for all $N > 0$. Thus, \(\|\bm{U}^N\|_{\mathbb{K}^h}^2 = \sum_{j \in \mathbb{Z}} \| q_h^N \|_{L^2(I_j)}^2 \le C_h(\bm{v}).\)
In the internal region $|j| \le N-1$, the truncated sequence $\bm{U}^N$ identically coincides with $u_h^*(x)$ generated by $\bm{u}_j^* = \bm{v}(-1)^j$. Let $q_h^*$ be the auxiliary variable of this infinite sequence. We then have $\sum_{j=-N+1}^{N-1} \| q_h^* \|_{L^2(I_j)}^2 \le C_h(\bm{v}).$
Because the sequence differs across adjacent cells only by a sign, which vanishes upon squaring, the local term $E_0 = \| q_h^* \|_{L^2(I_j)}^2$ is identical for all cells $j$. This yields $(2N-1)E_0 \le C_h(\bm{v})$. Letting $N \to \infty$, it holds that $E_0 = 0$. Consequently, $q_h^* = 0$ in each cell $I_j$, which yields
\[-\int_{I_j}
u_h^*\overline{v_x}\,\mathrm dx
+
\widetilde u_{h,j+\frac12}
\overline{v_{j+\frac12}^-}
-
\widetilde u_{h,j-\frac12}
\overline{v_{j-\frac12}^+}
=0,
\qquad
\forall v\in P^k(I_j).\]
Integrating the first term by parts, we obtain
\[\begin{aligned}
\int_{I_j}
(u_h^*)_x\overline v\,\mathrm dx
&+
\left(
\widetilde u_{h,j+\frac12}
-
(u_h^*)_{j+\frac12}^-
\right)
\overline{v_{j+\frac12}^-}-
\left(
\widetilde u_{h,j-\frac12}
-
(u_h^*)_{j-\frac12}^+
\right)
\overline{v_{j-\frac12}^+}
=0.
\end{aligned}\]
Choosing $\widetilde{u}_h = u_h^-$ and setting $v(x) = (x - x_{j-\frac{1}{2}}) (u_h^*)_x \in P^k(I_j)$, the equation reduces to
$\int_{I_j}
\left(x-x_{j-\frac12}\right)
\left|(u_h^*)_x\right|^2
\,\mathrm dx
=0$.
This implies $(u_h^*)_x \equiv 0$, meaning $u_h^*$ is a constant $c_j$ within each $I_j$. Choosing $v \equiv 1$ in the weak formulation yields 
\[\widetilde{u}_{h, j+\frac{1}{2}} - \widetilde{u}_{h, j-\frac{1}{2}} = (u_h^*)_{j+\frac{1}{2}}^- - (u_h^*)_{j-\frac{1}{2}}^- = c_j - c_{j-1} = 0.\]
Consequently, $\bm{v}(-1)^j = \bm c$ for all $j \in \mathbb{Z}$, forcing $\bm{v} = \mathbf{0}$. This contradiction proves $\mathcal{K}_1^h(\frac{\pi}{h})$ is positive definite.

Noting that $\overline{\mathcal{K}_1^h(\xi)} = \mathcal{K}_1^h(-\xi)$, the matrix $\mathcal{K}_1^h(-\frac{\pi}{h})$ is the conjugate of $\mathcal{K}_1^h(\frac{\pi}{h})$. As both are Hermitian, the positive definiteness extends to $\xi = -\frac{\pi}{h}$.
Returning to \eqref{Kronecker-sum structure}, the Lemma can be obtained directly.
\end{proof}

Lemma~\ref{lemma3.1} guarantees $(k+1)^d$ non-negative real eigenvalues for problem~\eqref{fully eigenvalue problem}. We now study the properties of these eigenvalues and their corresponding eigenvectors.
\begin{lemma} \label{lem:analytic_eigenvalues}
For the eigenvalue problems \eqref{fully eigenvalue problem}, there exist real analytic eigenpairs 
\(\{\sigma_{\bm m}(\boldsymbol{\xi}), \bm v_{\bm m}(\boldsymbol{\xi})\}_{\bm m\in\{1,\dots,k+1\}^d}\) 
whose eigenvectors form an \((\bm M^h)\)-orthonormal basis of \(\mathbb C^{(k+1)^d}\), and satisfy \(\bm v_{\bm m}(\boldsymbol{\xi})^H (\bm M^{h}) \bm v_{\bm n}(\boldsymbol{\xi}) = \delta_{\bm m,\bm n}\), 
where \(\delta_{\bm m,\bm n}:=\prod_{\ell=1}^d \delta_{m_\ell,n_\ell}\).
\end{lemma}
\begin{proof}
Considering the one-dimensional generalized eigenvalue problem \eqref{1d eigenvalues}, we introduce the change of variable \(\bm u(\xi)=(\bm{M}^h_1)^{1/2}\bm v(\xi)\). Thus, we obtain the standard eigenvalue problem $\widetilde{\mathcal{K}}_1^h(\xi) \bm{u}(\xi) = \sigma(\xi) \bm{u}(\xi)$, where $\widetilde{\mathcal{K}}_1^h(\xi) = (\bm{M}_1^h)^{-\frac{1}{2}} \mathcal{K}_1^h(\xi) (\bm{M}_1^h)^{-\frac{1}{2}}$.
    Since $\mathcal{K}_1^h(\xi)$ comprises finite linear combinations of $e^{\sqrt{-1} j \xi h}$, the matrix $\widetilde{\mathcal{K}}_1^h(\xi)$ remains a real analytic Hermitian family. 

Applying Kato's analytic perturbation theory \cite[Theorem 6.1]{MR1335452}, this one-parame\-ter Hermitian family admits real-analytic eigenvalues and
componentwise real-analytic complex-valued orthonormal
eigenvectors on an open interval containing $[-\pi/h,\pi/h]$.
The labels continue analytically through crossings, without
requiring ordering by size or periodic endpoint matching.
Transforming back preserves analyticity and gives
$\bm M_1^h$-orthonormality.

Using the Kronecker-sum representation
\eqref{Kronecker-sum structure}
and the identity \((A\otimes B)(C\otimes D)=(AC)\otimes(BD)\), we compute
\begin{align*}
\mathcal K^h(\boldsymbol{\xi})\bm v_{\bm m}(\boldsymbol{\xi})
&= \sum_{\ell=1}^d
\left(\bigotimes_{r=1}^{\ell-1} \bm M^h_1\bm v_{m_r}(\xi_r)\right)
\otimes \left(\mathcal{K}_1^h(\xi_\ell)\bm v_{m_\ell}(\xi_\ell)\right)
\otimes \left(\bigotimes_{r=\ell+1}^{d} \bm M^h_1\bm v_{m_r}(\xi_r)\right) \\
&= \sum_{\ell=1}^d \sigma_{m_\ell}(\xi_\ell)
(\bm M^h_1)^{\otimes d}
\left(\bigotimes_{r=1}^d \bm v_{m_r}(\xi_r)\right) = \sigma_{\bm m}(\boldsymbol{\xi})\,\bm{M}^h\bm v_{\bm m}(\boldsymbol{\xi}).
\end{align*}
Thus these tensor products give real-analytic eigenpairs
on a neighborhood of $\Pi_h^d$.

Finally, since \(\{\bm v_j(\xi_\ell)\}_{j=1}^{k+1}\) is an \(\bm M^h_1\)-orthonormal basis of \(\mathbb C^{k+1}\) for each \(\xi_\ell\), the tensor products form an \( \bm M^h\)-orthonormal basis of \(\mathbb C^{(k+1)^d}\)
\[
\bm{v}_{\bm m}^H \bm{M}^h\bm{v}_{\bm n}=\left(\bigotimes_{\ell=1}^d \bm v_{m_\ell}(\xi_\ell)\right)^H
\bm M^{h}
\left(\bigotimes_{\ell=1}^d \bm v_{n_\ell}(\xi_\ell)\right)
= \prod_{\ell=1}^d \delta_{m_\ell,n_\ell}
= \delta_{\bm m,\bm n}.
\]
There are exactly \((k+1)^d\) such pairs, which equals the dimension of the space.
\end{proof}

Having established the real analyticity of these eigenmodes, we now proceed to identify the physical and spurious branches.
\begin{lemma}\label{fully proposition1}
There exists only one eigenvalue $\sigma_{\bm{ph}}(\bm\xi)$ and corresponding eige\-nvector $\bm{v}_{\bm{ph}}(\bm\xi)$ satisfying \(\sigma_{\bm{ph}}(\boldsymbol 0)=0,\, 
        \bm{v}_{\bm{ph}}(\boldsymbol 0)\propto \bm{e}_0^{\otimes d}.\)
The remaining $(k+1)^d-1$ spurious branches $(\sigma_{\bm{m}}(\bm\xi), \bm{v}_{\bm m}(\bm\xi)),\bm{m}\ne\bm{ph}$ satisfy \(\sigma_{\bm m}(\bm 0)>0.\)
\end{lemma}
\begin{proof}
    At $\xi = 0$, the matrix 
    $ \mathcal{K}_1^h(0) = \bm{K}^h_0 + \bm{K}^h_{-1} + \bm{K}^h_{+1}$.
   At zero frequency, $\bm e_0$ represents a constant mode with matching
traces across cell interfaces. The volume and interface terms in the
LDG gradient relation therefore cancel, giving $q_h=0$ and hence
$\mathcal K_1^h(0)\bm e_0=\bm0$.
    Thus, the system admits a zero eigenvalue $\sigma(0) = 0$ with eigenvector $\bm{v}(0) \propto \bm{e}_0$. We denote this specific  eigenpair as $\{\sigma_{ph}, \bm{v}_{ph}\}$.

    To determine the algebraic multiplicity of the zero eigenvalue $\sigma_{ph}(0)=0$, we examine the quadratic form associated with the stiffness operator.  A straightforward calculation reveals that
    \( \bm{v}^H \mathcal{K}_1^h(0) \bm{v} = \int_{I_j} |q_h(\bm{v})|^2 \mathrm{d}x \ge 0, \, \forall \bm{v} \in \mathbb{C}^{k+1}. \)
    The equality holds if and only if $\|q_h\|_{L^2(I_j)} = 0$, which implies that $u_h$ is a constant. Therefore, $\operatorname{Ker}(\mathcal{K}_1^h(0)) = \operatorname{span}\{\bm{e}_0\}$. 
    Given that $\mathcal{K}_1^h(0)$ is Hermitian positive semi-definite and $\bm{M}_1^h$ is positive definite, the zero eigenvalue $\sigma_{ph}(0)$ has algebraic multiplicity 1. It follows that the remaining $k$ spurious eigenvalues satisfy $\sigma_{sp,i}(0) > 0$.

    For the \(d\)-dimensional system, by the additive decomposition of Lemma~\ref{lem:analytic_eigenvalues}, the branch \(\bm m=\bm{ph}\) satisfies
\(
\sigma_{\bm{ph}}(\bm 0)=\sum_{\ell=1}^d \sigma_{ph}(0)=0,\,
\bm v_{\bm{ph}}(\bm 0)=\bigotimes_{\ell=1}^d \bm v_{ph}(0)\propto \mathbf e_0^{\otimes d}.
\)

For \(\bm m\ne\bm{ph}\), at least one component \(m_\ell\ne ph\) is spurious. Since \(\sigma_{m_\ell}(0)>0\) for the spurious index and \(\sigma_{m_r}(0)\ge 0\) for all \(r\ne\ell\), we get
\(
\sigma_{\bm m}(\bm 0)=\sum_{\ell=1}^d \sigma_{m_\ell}(0)>0.
\)
The uniqueness of this branch follows from the uniqueness of the one-dimensional physical branch and the tensor-product construction. This completes the proof.
\end{proof}

The following Lemma establishes the symmetry of the physical eigenvalue, which is fundamental for the subsequent asymptotic analysis.
\begin{lemma}\label{fully even function}
The analytic physical eigenvalue of the system \eqref{fully eigenvalue problem} is an even function with respect to \(\boldsymbol{\xi}\), satisfying
\(
\sigma_{\boldsymbol m}(\boldsymbol{\xi}) = \sigma_{\boldsymbol m}(-\boldsymbol{\xi}), \,
\boldsymbol m=\bm{ph}.
\)
\end{lemma}
\begin{proof}
Define
\(
\widetilde{\sigma}_{\bm{m}}(\boldsymbol{\xi})
:=
\sigma_{\bm{m}}(-\boldsymbol{\xi}),\bm m=\bm{ph}.
\)
Using
\(
\mathcal K^h(-\boldsymbol{\xi})
=
\overline{\mathcal K^h(\boldsymbol{\xi})},
\)
together with the fact that \(\bm M^h\) and
\(\sigma_{\bm{m}}(-\boldsymbol{\xi})\) are real, we obtain
\begin{align*}
0
=
\det\!\left(
\mathcal K^h(-\boldsymbol{\xi})
-
\sigma_{\bm{m}}(-\boldsymbol{\xi})\bm M^h
\right)=
\overline{
\det\!\left(
\mathcal K^h(\boldsymbol{\xi})
-
\sigma_{\bm{m}}(-\boldsymbol{\xi})\bm M^h
\right)
}.
\end{align*}
Hence,
\(
\det\!\left(
\mathcal K^h(\boldsymbol{\xi})
-
\widetilde{\sigma}_{\bm{m}}(\boldsymbol{\xi})\bm M^h
\right)
=0,
\)
so \(\widetilde{\sigma}_{\bm{m}}(\boldsymbol{\xi})\) is an
eigenvalue of the generalized eigenvalue problem at
\(\boldsymbol{\xi}\).
At \(\boldsymbol{\xi}=\boldsymbol 0\),
\(
\widetilde{\sigma}_{\bm{m}}(\boldsymbol 0)
=
\sigma_{\bm{m}}(\boldsymbol 0)
=
0.
\)
Moreover, by Lemma~\ref{fully proposition1}, the zero eigenvalue
at the origin is simple, while all other eigenvalues are strictly
positive. Therefore, there exists a neighborhood
\(\mathcal U\) of \(\boldsymbol 0\) in which the physical eigenvalue
is the unique eigenvalue branch issuing from zero. It follows that
\(
\widetilde{\sigma}_{\bm{m}}(\boldsymbol{\xi})
=
\sigma_{\bm{m}}(\boldsymbol{\xi}),
\,
\boldsymbol{\xi}\in\mathcal U.
\)

Since both functions are real analytic and agree on the nonempty open
set \(\mathcal U\), the identity theorem for real-analytic functions
implies
\(
\sigma_{\bm{m}}(-\boldsymbol{\xi})
=
\sigma_{\bm{m}}(\boldsymbol{\xi})
\)
throughout the interior of \(\Pi_h^d\). The equality extends to the
boundary of \(\Pi_h^d\) by continuity.
\end{proof}

We are now well-positioned to verify the low-frequency consistency of the continuous wave equation under our fully discrete framework \eqref{eq:global_fully_discrete}.

\begin{lemma}\label{fully physical eigenvalue limit}
The physical temporal frequency satisfies
\(
\lim_{\boldsymbol\xi\to\boldsymbol0}
\frac{\omega_{\bm{ph}}(\boldsymbol\xi)}
{|\boldsymbol\xi|}
=1\).
\end{lemma}
\begin{proof}
    We consider the Taylor expansion of the one-dimensional problem at $\xi=0$
\begin{equation*}
    \sigma_{ph}(\xi) = \sigma^{(2)} \xi^2 + \frac{\sigma_{ph}^{(4)}(\eta)}{24} \xi^4, \qquad \text{where } \sigma^{(2)} = \frac{1}{2} \sigma_{ph}''(0),
\end{equation*}
for some intermediate $\eta$ between $0$ and $\xi$.

Let the eigenvector be normalized such that $\bm{v}_{ph}(\xi)^H \bm{M}_1^h \bm{v}_{ph}(\xi) = 1$, which implies
$\sigma_{ph}(\xi) = \bm{v}_{ph}^H(\xi) \mathcal{K}_1^h(\xi) \bm{v}_{ph}(\xi)$.
Taking the first-order derivative with respect to $\xi$ yields
\begin{equation}\label{first-order total derivative}
    \sigma_{ph}'(\xi) = \bm{v}_{ph}^H \mathcal{K}_1^{h \prime}(\xi) \bm{v}_{ph} + \bm{v}_{ph}^H \mathcal{K}_1^h(\xi) \bm{v}_{ph}' + (\bm{v}_{ph}')^H \mathcal{K}_1^h(\xi) \bm{v}_{ph}.
\end{equation}
Utilizing the Hermitian property of $\mathcal{K}_1^h(\xi)$, the last two terms of equation \eqref{first-order total derivative} simplify to $\sigma_{ph} (\bm{v}_{ph}^H \bm{M}_1^h \bm{v}_{ph}' + (\bm{v}_{ph}')^H \bm{M}_1^h \bm{v}_{ph})$. Differentiating the normalization condition $\bm{v}_{ph}^H \bm{M}_1^h \bm{v}_{ph} = 1$ gives $\bm{v}_{ph}^H \bm{M}_1^h \bm{v}_{ph}' + (\bm{v}_{ph}')^H \bm{M}_1^h \bm{v}_{ph} = 0$, thus confirming the result $\sigma_{ph}'(\xi) = \bm{v}^H_{ph}(\xi) \mathcal{K}_1^{h \prime}(\xi) \bm{v}_{ph}(\xi)$.
Differentiating once more with respect to $\xi$ yields
\begin{equation*}
    \sigma_{ph}''(\xi) = \bm{v}_{ph}^H \mathcal{K}_1^{h \prime\prime}(\xi) \bm{v}_{ph} + \bm{v}_{ph}^H \mathcal{K}_1^{h \prime}(\xi) \bm{v}_{ph}' + (\bm{v}_{ph}')^H \mathcal{K}_1^{h \prime}(\xi) \bm{v}_{ph}.
\end{equation*}
At $\xi=0$, we have
$\bm{v}_{ph}(0)=h^{-1/2}\bm{e}_0$.
For $k\geq1$, let $\bm{e}_1$ denote the coefficient vector
of the local linear basis function $2(x-x_j)/h$.
The alternating-flux LDG matrices satisfy
\(
\mathcal{K}_1^h(0)\left(\frac{\sqrt{-1}h}{2}\bm{e}_1\right)
=-\mathcal{K}_1^{h\prime}(0)\bm{e}_0.
\)
Differentiating the generalized eigenvalue equation at $\xi=0$,
using $\sigma_{ph}(0)=\sigma_{ph}'(0)=0$, and fixing the phase by
$\bm{v}_{ph}(0)^H\bm{M}_1^h\bm{v}_{ph}'(0)=0$, we therefore obtain
\(
\bm{v}_{ph}'(0)=\frac{\sqrt{-h}}{2}\bm{e}_1.
\)
Moreover, direct calculation from the LDG matrices gives
\(
\bm{e}_0^H\mathcal{K}_1^{h\prime\prime}(0)\bm{e}_0
=2h(k+1)^2,
\,
\bm{e}_0^H\mathcal{K}_1^{h\prime}(0)\bm{e}_1
=2\sqrt{-1}k(k+2).
\)
Consequently,
\[
\begin{aligned}
\sigma_{ph}''(0)
&=\bm{v}_{ph}(0)^H\mathcal{K}_1^{h\prime\prime}(0)\bm{v}_{ph}(0)
 +2\operatorname{Re}\!\left(
 \bm{v}_{ph}(0)^H\mathcal{K}_1^{h\prime}(0)\bm{v}_{ph}'(0)
 \right)=2.
\end{aligned}
\]
For $k=0$, the same identity follows directly from
$\sigma_{ph}(\xi)=4h^{-2}\sin^2(h\xi/2)$.
Since $\sigma_{ph}$ is even and analytic, with
$\sigma_{ph}(0)=0$ and $\sigma_{ph}''(0)=2$,
Taylor's theorem yields constants $C_\sigma>0$ and $\delta>0$
such that
\(
\left|\sigma_{ph}(\xi)-\xi^2\right|
\leq C_\sigma|\xi|^4,
\, |\xi|\leq\delta.
\)
Reducing $\delta$ if necessary, we may assume that
$\sigma_{ph}(\xi)\leq2\xi^2$ and
$\Delta t\sqrt{\sigma_{ph}(\xi)}\leq1$.
Using
$0\leq\arcsin z-z\leq z^3$ for $0\leq z\leq1/2$,
we obtain
\[
\begin{aligned}
\left|\omega_{ph}(\xi)-|\xi|\right|
\leq
\left|\sqrt{\sigma_{ph}(\xi)}-|\xi|\right|
+\frac{(\Delta t)^2}{4}\sigma_{ph}(\xi)^{\frac{3}{2}}\leq
\left(C_\sigma+\frac{(\Delta t)^2}{\sqrt{2}}\right)|\xi|^3,
\, 0<|\xi|\leq\delta.
\end{aligned}
\]
Consequently,
\(
\left|\frac{\omega_{ph}(\xi)}{|\xi|}-1\right|
\leq
\left(C_\sigma+\frac{(\Delta t)^2}{\sqrt{2}}\right)|\xi|^2,
\)
which implies
$\displaystyle
\lim_{\xi\to0^\pm}\frac{\omega_{ph}(\xi)}{\xi}=\pm1$.

We now extend the one-dimensional result to \(d\) dimensions. By Lemma~\ref{lem:analytic_eigenvalues}, the \(d\)-dimensional eigenvalue satisfies 
\(
\sigma_{\bm{ph}}(\boldsymbol{\xi}) = \sum_{\ell=1}^d \sigma_{ph}(\xi_\ell)
\).
From the one-dimensional result established above, for each direction \(\ell\), there exist constants \(C>0\) and \(\delta>0\) such that, whenever \(|\xi_\ell|<\delta\),
\(
\left|\sigma_{ph}(\xi_\ell) - \xi_\ell^2\right| \le C |\xi_\ell|^4.
\)
Summing over \(\ell=1,\dots,d\), for all \(\boldsymbol{\xi}\) with \(|\boldsymbol{\xi}|<\delta\), we obtain
\(
\left|\sigma_{\bm{ph}}(\boldsymbol{\xi}) - |\boldsymbol{\xi}|^2\right|
\le \sum_{\ell=1}^d \left|\sigma_{ph}(\xi_\ell) - \xi_\ell^2\right|
\le  C_d |\boldsymbol{\xi}|^4.
\)
Hence
\(
\sigma_{\bm{ph}}(\boldsymbol{\xi}) = |\boldsymbol{\xi}|^2 + R(\boldsymbol{\xi}), \,|R(\boldsymbol{\xi})| \le C_d |\boldsymbol{\xi}|^4,
\)
for all sufficiently small \(\boldsymbol{\xi}\in\Pi_h^d\).

Substituting this into the definition of the \(d\)-dimensional temporal frequency
\(
\omega_{\bm{ph}}(\boldsymbol{\xi}) = \frac{2}{\Delta t}\arcsin\!\left(\frac{\Delta t}{2}\sqrt{\sigma_{\bm{ph}}(\boldsymbol{\xi})}\right),
\)
and using the inequality \(|\arcsin(z)-z| \le C |z|^3\) for sufficiently small \(z\), we get
\(
\omega_{\bm{ph}}(\boldsymbol{\xi}) = |\boldsymbol{\xi}| + \widetilde R(\boldsymbol{\xi}), \, |\widetilde R(\boldsymbol{\xi})| \le C'_d |\boldsymbol{\xi}|^3,
\)
for all sufficiently small \(|\boldsymbol{\xi}|\). The identity \(\nabla^2\sigma_{\bm{ph}}(\boldsymbol{0}) = 2\,\bm I_d\) follows directly from the one-dimensional consistency property \(\sigma_{ph}''(0) = 2\).
This completes the proof.
\end{proof}
Theorem~\ref{thm:spectral_structure} follows directly from the preceding Lemmas.

\subsection{Proof of Theorem~\ref{thm:fully_group_velocity_limits}}\label{subsection 3.1}
Before proving the main Theorem, we first establish the simplicity property for the one-dimensional eigenvalues.
\begin{lemma}\label{simple}
The one-dimensional eigenvalues \(\sigma_{i}(\pm\frac{\pi}{h}),i=1, \cdots, k+1\) are simple.
\end{lemma}
\begin{proof}
The case $k=0$ is immediate. For $k\ge1$, use the Legendre
basis on $[-1,1]$, which preserves the generalized eigenvalues.
Set $\theta=h\xi$ and $z=e^{-\sqrt{-1}\theta}$. The alternating
fluxes give, for $0\le a,b\le k$,
\[
(\bm M_1^h)_{ab}=\frac{h\delta_{ab}}{2a+1},\qquad
(\bm B_k(z))_{ab}
=1-(-1)^a z-2\mathbf1_{\{a>b,\ a-b\text{ odd}\}},
\]
and $\mathcal K_1^h=\bm B_k^H(\bm M_1^h)^{-1}\bm B_k$.
Thus the generalized eigenvalues are those of $\bm C^H\bm C$,
where $\bm C=(\bm M_1^h)^{-1/2}\bm B_k(\bm M_1^h)^{-1/2}$.
For $0<\theta\le\pi$, writing $\bm e_j$ for the coordinate vectors,
\[
\begin{aligned}
\bm B_k\!\!\left(\frac{1+z}{1-z}\bm e_0+\bm e_1\right)=2\bm e_0,
\bm B_k(\bm e_{j+1}-\bm e_{j-1})=2\bm e_j\,(1\le j<k),
\bm B_k(\bm e_k-\bm e_{k-1})=2\bm e_k.
\end{aligned}
\]
These identities verify all columns of $\bm B_k^{-1}$, hence
$\bm J=\bm C^{-1}=(\bm M_1^h)^{1/2}\bm B_k^{-1}(\bm M_1^h)^{1/2}$
has the following nonzero entries
\[
J_{00}
=
-\frac{\sqrt{-1}h}{2}\cot\left(\frac{\theta}{2}\right),
\qquad
J_{kk}=\frac{h}{2(2k+1)},
\]
\[
J_{n,n+1}
=
-\frac{h}{2\sqrt{(2n+1)(2n+3)}},
\quad
J_{n+1,n}
=
\frac{h}{2\sqrt{(2n+1)(2n+3)}},
\quad n=0,\dots,k-1.
\]
Let
\(
\bm S(\theta):=\bm J(\theta)^H\bm J(\theta).
\)
After ordering the indices as
\(
0,2,4,\dots,\quad\dots,5,3,1,
\)
let \(\bm P\) be the corresponding permutation matrix. Then
\(\bm P^T\bm S(\pi)\bm P\) is a real symmetric tridiagonal matrix
with nonzero off-diagonal entries, and hence has simple eigenvalues.
Since the eigenvalues of \(\bm S(\pi)\) are the reciprocals of those
of \(\bm C(\pi)^H\bm C(\pi)\), the generalized eigenvalues of
\(
\bigl(\mathcal K_1^h(\pm\pi/h),\bm M_1^h\bigr)
\)
are simple.
\end{proof}
The preceding Lemma gives the needed simplicity. We now prove Theorem~\ref{thm:fully_group_velocity_limits}.
\begin{proof}
First, the low-frequency expansion in Lemma~\ref{fully physical eigenvalue limit} yields \eqref{physical velocity}.

Applying the chain rule yields
\begin{equation}\label{group velocity component formula}
    \bigl(\bm V_g^{\bm m}(\boldsymbol{\xi})\bigr)_\ell
= \frac{\partial_{\xi_\ell}\sigma_{\bm m}(\boldsymbol{\xi})}{\sqrt{4-(\Delta t)^2\sigma_{\bm m}(\boldsymbol{\xi})}\sqrt{\sigma_{\bm m}(\boldsymbol{\xi})}},\quad \ell=1,\dots,d,\quad\bm m=\bm{ph}.
\end{equation}
For
\(
\boldsymbol\xi_*
=
\frac{\pi}{h}\boldsymbol s,
\,
\boldsymbol s\in\{\pm1\}^d,
\) we first prove that \(\partial_{\xi_\ell}\sigma_{\boldsymbol m}(\boldsymbol{\xi}_*)=0\) for each \(\ell\). 
Applying the formula from the proof of Lemma~\ref{fully physical eigenvalue limit} for the generalized eigenvalue problem gives
\begin{equation}
    \partial_{\xi_\ell}\sigma_{\boldsymbol m}(\boldsymbol{\xi}_*)
    =
    \boldsymbol v_{\boldsymbol m}(\boldsymbol{\xi}_*)^H
    \left(\partial_{\xi_\ell}\mathcal K^h(\boldsymbol{\xi}_*)\right)
    \boldsymbol v_{\boldsymbol m}(\boldsymbol{\xi}_*).
\end{equation}

From the Kronecker-sum structure,
\begin{equation}
    \partial_{\xi_\ell}\mathcal K^h(\boldsymbol{\xi}_*)
    =
    (\bm {M}^h_1)^{\otimes(\ell-1)}
    \otimes
    (\mathcal{K}^h_1)'\left(s_\ell\frac{\pi}{h}\right)
    \otimes
    (\bm {M}^h_1)^{\otimes(d-\ell)}.
\end{equation}
In one dimension,
\(
(\mathcal{K}^h_1)'\left(s_\ell\frac{\pi}{h}\right)=\sqrt{-1}h\left(K_{-1}^h-K_{+1}^h\right),
\)
and \(K_{-1}^h-K_{+1}^h\) is real skew-symmetric since \((K_{-1}^h)^T=K_{+1}^h\). 
By the tensor-product structure and
Lemma~\ref{simple},
\(
\boldsymbol v_{\boldsymbol m}(\boldsymbol{\xi}_*)
=
\bigotimes_{\ell=1}^d
\bm v_{ph}\left(s_\ell\frac{\pi}{h}\right),
\,
\bm v_{ph}\left(s_\ell\frac{\pi}{h}\right)\in\mathbb R^{k+1}.
\)
Thus,
\[
\partial_{\xi_\ell}\sigma_{\boldsymbol m}(\boldsymbol{\xi}_*)
=
\left[
\prod_{r\ne\ell}
(\bm v_{ph}^T \bm M_1^h \bm v_{ph})
\left(s_r\frac{\pi}{h}\right)
\right]
(\bm v_{ph}^T(\mathcal K_1^h)'\bm v_{ph})
\left(s_\ell\frac{\pi}{h}\right).
\]
By the normalization
\(
\bm v_{ph}^T\bm M_1^h\bm v_{ph}=1
\)
and the skew-symmetry of \(K_{-1}^h-K_{+1}^h\), the right-hand side
vanishes.
Hence
\(
\partial_{\xi_\ell}\sigma_{\boldsymbol m}(\boldsymbol{\xi}_*)=0,
\, \ell=1,\dots,d.
\)
Substituting this into \eqref{group velocity component formula} gives
\(
\bigl(\bm V_g^{\boldsymbol m}(\boldsymbol{\xi}_*)\bigr)_\ell=0,
\, \ell=1,\dots,d,
\)
since the denominator is non-zero under the non-critical assumption \(4-(\Delta t)^2\sigma_{\boldsymbol m}(\boldsymbol{\xi}_*)>0\). 


We finally consider the stated critical case for
\(
k=0,
\,
\sqrt d\,\Delta t=h.
\)
Let
\(
\boldsymbol a\in[0,\infty)^d,
\,
|\boldsymbol a|=1.
\)
Define \(
\boldsymbol{\xi}(f)
=
\boldsymbol{\xi}_*
-
f\,\operatorname{diag}(\boldsymbol{s})\boldsymbol{a},
\,
f>0.
\)
Then
\(
\xi_\ell(f)
=
s_\ell\left(\frac{\pi}{h}-fa_\ell\right),
\)
and, for sufficiently small \(f\),
\(\boldsymbol\xi(f)\) lies in \(\Pi_h^d\setminus\{\bm\xi_*\}\).
For \(k=0\), the physical eigenvalue satisfies
\[
\sigma_{\bm{ph}}(\boldsymbol{\xi}(f))
=
\frac{4}{h^2}
\sum_{r=1}^d
\cos^2\left(\frac{fha_r}{2}\right).
\]
Moreover,
\(
\partial_{\xi_\ell}\sigma_{\bm{ph}}
(\boldsymbol{\xi}(f))
=
\frac{2}{h}s_\ell\sin(fha_\ell),
\)
and, since \((\Delta t)^2=h^2/d\),
\[
4-(\Delta t)^2
\sigma_{\bm{ph}}(\boldsymbol{\xi}(f))
=
\frac{4}{d}
\sum_{r=1}^d
\sin^2\left(\frac{fha_r}{2}\right).
\]
Consequently,
\[
\left(
\bm V_g^{\bm{ph}}(\boldsymbol{\xi}(f))
\right)_\ell
=
\frac{
\frac{\sqrt d}{2}s_\ell\sin(fha_\ell)
}{
\left(
\sum_{r=1}^d
\cos^2\left(\frac{fha_r}{2}\right)
\right)^{1/2}
\left(
\sum_{r=1}^d
\sin^2\left(\frac{fha_r}{2}\right)
\right)^{1/2}
}.
\]
Using \(|\boldsymbol a|=1\), we obtain
\(
\lim_{f\to0^+}
\left(
\bm V_g^{\bm{ph}}(\boldsymbol{\xi}(f))
\right)_\ell
=
s_\ell a_\ell.
\)
Hence,
\[
\lim_{f\to0^+}
\bm V_g^{\bm{ph}}
\left(
\boldsymbol{\xi}_*
-
f\,\operatorname{diag}(\boldsymbol{s})\boldsymbol{a}
\right)
=
\operatorname{diag}(\boldsymbol{s})\boldsymbol{a}.
\]

\end{proof}

\subsection{Proof of Theorem~\ref{non_uniform_observability}}\label{subsection 3.2}
Before proving our main Theorem, we first derive the bound for the total energy of the initial data given by \eqref{eq:initial_data_construction}.
\begin{lemma}\label{energy_lower_bound}

For the discrete initial data $(\bm{U}^0, \bm{U}^1)$ constructed in \eqref{eq:initial_data_construction}, there exist positive constants $C_{1}$ and $C_{2}$ independent of $h$ such that \(C_{1} \leq \mathcal{E}^{0}_h(\bm{U}^0, \bm{U}^1) \leq C_{2}.\)
\end{lemma}
\begin{proof}
By Parseval's identity, the energy \eqref{eq:discrete_energy} can be expressed as
\begin{equation*} \label{eq:discrete_fourier_energy_d}
    \mathcal{E}_h^{0} = \frac{1}{2(2\pi)^d h^d} \int_{\Pi_h^d} ( \langle \bm{M}^h \frac{\widehat{\bm{U}}^{1}-\widehat{\bm{U}}^0}{\Delta t}, \frac{\widehat{\bm{U}}^{1}-\widehat{\bm{U}}^0}{\Delta t} \rangle + \operatorname{Re}\langle \mathcal K^h(\boldsymbol{\xi})\widehat{\bm{U}}^{1}, \widehat{\bm{U}}^0 \rangle )(\boldsymbol{\xi}) \,\mathrm{d}\boldsymbol{\xi}.
\end{equation*}
Denoting $\bm m= \bm{ph}$ in this proof and applying $\widehat{\bm{U}}^1 = e^{-\sqrt{-1}\omega_{\bm m}(\boldsymbol{\xi})\Delta t}\widehat{\bm{U}}^0$ together with the eigenvalue relation, we simplify
\begin{align}\label{eq:4.5_d}
    \mathcal{E}^{0}_h &= \frac{1}{2(2\pi)^d h^d} \int_{\Pi_h^d} \sigma_{\bm m}(\boldsymbol{\xi}) \bigl[ 1 + \cos(\omega_{\bm m}(\boldsymbol{\xi})\Delta t) \bigr] \langle \bm{M}^h\widehat{\bm{U}}^{0}, \widehat{\bm{U}}^0 \rangle(\boldsymbol{\xi}) \,\mathrm{d}\boldsymbol{\xi}.
\end{align}
Under Assumption~\ref{assump:observability}, there exist positive constants \(C^1_{s}\) and \(C^2_{s}\) such that
\begin{equation} \label{cos_bound_d}
    C^1_{s} \geq 1 + \cos(\omega_{\bm m}(\boldsymbol{\xi})\Delta t) \geq C^2_{s}, \quad \forall \boldsymbol{\xi} \in \Omega_{\boldsymbol{\xi}_c,\rho} =  \prod_{\ell=1}^d [\xi_{c\ell} - \rho, \xi_{c\ell} + \rho] \subset \Pi_h^d.
\end{equation}

Using the definition of $\widehat{\bm{U}}^0$ and  $\langle \bm{M}^h \bm v_{\bm m}(\boldsymbol{\xi}), \bm v_{\bm m}(\boldsymbol{\xi}) \rangle = 1$, we have
\begin{align} \label{hatU_d}
    \langle \bm{M}^h\widehat{\bm{U}}^{0}, \widehat{\bm{U}}^0 \rangle(\boldsymbol{\xi}) 
    \!=\! h^{(1+\gamma) d} |r(\boldsymbol{\xi})|^2 |\chi_\rho(\boldsymbol{\xi})|^2 \!\langle \bm{M}^h \bm v_{\bm m}(\boldsymbol{\xi}), \bm v_{\bm m}(\boldsymbol{\xi}) \rangle \!\!=\! \!\frac{h^{(1+\gamma) d}}{\sigma_{\bm m}(\boldsymbol{\xi})} |\chi_\rho(\boldsymbol{\xi})|^2.
\end{align}

Substituting and noting that the support of $\chi_\rho$ restricts the integral to $\Omega_{\boldsymbol{\xi}_c,\rho}$, we obtain
\begin{align*}
    \frac{C^2_s}{2(2\pi)^d } h^{\gamma d} \int_{\Omega_{\boldsymbol{\xi}_c,\rho}} |\chi_\rho(\boldsymbol{\xi})|^2 \,\mathrm{d}\boldsymbol{\xi}
    &\leq \mathcal{E}_h^{0} \leq \frac{C^1_s}{2(2\pi)^d } h^{\gamma d} \int_{\Omega_{\boldsymbol{\xi}_c,\rho}} |\chi_\rho(\boldsymbol{\xi})|^2\,\mathrm{d}\boldsymbol{\xi}.
\end{align*}

Finally, performing the change of variables $\bm g = (\bm\xi - \bm\xi_c)/\rho$ with $\mathrm{d}\bm\xi = \rho^d \mathrm{d}\bm g$, and noting that the scaling relationship dictates $\rho = h^{-\gamma}$, we find
\begin{align*}
    \frac{C^2_{s}}{2(2\pi)^d} \|\psi\|_{L^2([-1,1])}^{2d} 
   \leq \mathcal{E}_h^{0} \leq \frac{C^1_{s}}{2(2\pi)^d} \|\psi\|_{L^2([-1,1])}^{2d}.
\end{align*}
Setting $C_1 = \frac{C^2_{s}}{2(2\pi)^d} \|\psi\|_{L^2([-1,1])}^{2d}$ and $C_2 = \frac{C^1_{s}}{2(2\pi)^d} \|\psi\|_{L^2([-1,1])}^{2d}$ establishes the desired constants, which are positive and independent of $h$.
\end{proof}

    We now establish the exponential decay of the observed energy. Set \(d_c:=\frac12\operatorname{dist}
(\bm x_c,\mathbb R^d\setminus(-1,1)^d)>0\) and
\(
\mathcal J_\Omega^+
:=\{\bm j+\bm\alpha:\bm j\in\mathcal J_\Omega,
\ \bm\alpha\in\{-1,0,1\}^d\}.
\)
For \(0<h\le d_c/\sqrt d\), every \(\bm j\in\mathcal J_\Omega^+\)
satisfies \(|\bm x_{\bm j}-\bm x_c|\ge2d_c-\sqrt d\,h\ge d_c\).
The following estimates are taken in this range of \(h\).
\begin{lemma}
\label{nonstationary_bound}
Fix \(T>0\). For any integer \(p\ge1\), sufficiently small
\(h\), \(\bm j\in\mathcal J_\Omega^+\), and \(t_n\in[0,T]\),
write
\(\Phi_{\bm j}^n(\bm\xi)
:=(\bm x_{\bm j}-\bm x_c)\cdot\bm\xi
-\omega_{\bm m}(\bm\xi)t_n\).
Then \(\bm u_{\bm j}^n\) satisfies
\begin{equation}
    \bm{u}_{\bm j}^n = \frac{1}{(2\pi)^d} \int_{\Omega_{\bm\xi_c,\rho}}\!\!\!\!\!\!\!\! e^{\sqrt{-1} \Phi_{\bm j}^n(\bm\xi)} (\mathcal L^*)^p[\mathcal A(\bm\xi)] \,\mathrm{d}\bm\xi\, \text{with}\,\mathcal L^*f:=-\nabla_{\boldsymbol{\xi}}\cdot( \frac{\nabla_{\boldsymbol{\xi}}\Phi_{\bm j}^n}{\sqrt{-1}\,|\nabla_{\boldsymbol{\xi}}\Phi_{\bm j}^n|^2} f ),
\end{equation}
where
\(\mathcal A(\boldsymbol{\xi}):=r(\boldsymbol{\xi})h^{(1+\gamma)d/2}\chi_\rho(\boldsymbol{\xi})\bm v_{\bm m}(\boldsymbol{\xi})\),
and \(\bm m=\bm{ph}\).
\end{lemma}

\begin{proof}
Denote \(\bm m=\bm{ph}\) in this proof. Due to the modal purity of the initial data \eqref{eq:initial_data_construction} and the linearity of the time-stepping scheme, the temporal evolution of the wave packet is given by
\begin{equation}
    \widehat{\bm{U}}^n(\boldsymbol{\xi}) = e^{-\sqrt{-1} \omega_{\bm m}(\boldsymbol{\xi}) t_n} \widehat{\bm{U}}^0(\boldsymbol{\xi}), \quad \forall n \geq 0.
\end{equation}
Using the inverse SDFT, the solution in the physical space at cell \(\bm I_{\bm j}\) and time \(t_n\) is
\begin{equation}
    \bm{u}_{\bm j}^n = \frac{1}{(2\pi)^d} \int_{\Pi_h^d} \widehat{\bm{U}}^n(\bm{\xi}) e^{\sqrt{-1} \bm{\xi} \cdot \bm{x}_{\bm{j}}} \,\mathrm{d}\bm{\xi}
    = \frac{1}{(2\pi)^d} \int_{\Omega_{\bm\xi_c,\rho}} \mathcal A(\bm\xi) e^{\sqrt{-1}\Phi_{\bm j}^n(\bm\xi)} \,\mathrm{d}\bm\xi,
\end{equation}
where \(\mathcal A(\bm\xi) = r(\boldsymbol{\xi}) h^{(1+\gamma) d/2} \chi_{\rho}(\boldsymbol{\xi})  \bm{v}_{\bm m}(\boldsymbol{\xi})\) and \(\Phi_{\bm j}^n(\bm\xi) = (\bm{x_j}-\bm x_c)\cdot \bm{\xi} - \omega_{\bm m}(\bm\xi)t_n\).

The wave packet is concentrated near \((\frac{\pi}{h},\dots,\frac{\pi}{h})\), where the group velocity vanishes. By continuity, for sufficiently small \(h\), \(|\bm V^{\bm m}_g(\bm\xi)|\) remains small in \(\Omega_{\bm{\xi}_c,\rho}\). Thus, for \(\bm j\in\mathcal J_\Omega^+\) and any \(t_n \in [0,T]\),
\(
|\nabla_{\boldsymbol\xi}\Phi_{\bm j}^n(\bm\xi)| \ge |\bm{x_j}-\bm x_c| - |\bm V^{\bm m}_g(\bm\xi)|T > \frac{|\bm{x_j}-\bm x_c|}{2} > 0.
\)

Define the differential operator $\mathcal L$ and its adjoint $\mathcal{L}^*$ acting on a function $f$ as
\[
\mathcal L f := \frac{\nabla_{\boldsymbol{\xi}}\Phi_{\bm j}^n}{\sqrt{-1}\,|\nabla_{\boldsymbol{\xi}}\Phi_{\bm j}^n|^2} \cdot \nabla_{\boldsymbol{\xi}} f,\qquad\mathcal L^* f := -\nabla_{\boldsymbol{\xi}}\cdot( \frac{\nabla_{\boldsymbol{\xi}}\Phi_{\bm j}^n}{\sqrt{-1}\,|\nabla_{\boldsymbol{\xi}}\Phi_{\bm j}^n|^2} f).
\]
Observe that
\(
\mathcal L(e^{\sqrt{-1}\Phi_{\bm j}^n}) = e^{\sqrt{-1}\Phi_{\bm j}^n}.
\)
Integrating by parts iteratively \(p\) times yields
\[
\bm{u}_{\bm j}^n = \frac{1}{(2\pi)^d} \int_{\Omega_{\bm\xi_c,\rho}} e^{\sqrt{-1} \Phi_{\bm j}^n(\bm\xi)} (\mathcal L^*)^p[\mathcal A(\bm\xi)] \,\mathrm{d}\bm\xi.
\]
This completes the proof.
\end{proof}

\begin{lemma}\label{nonstationary_bound2}
For every integer \(p\ge1\) and $h$ sufficiently small, there
exist nonnegative functions
\(
D_{p,\boldsymbol\alpha}^*(\boldsymbol\xi),
\, |\boldsymbol\alpha|\le p,
\)
such that
\begin{equation}\label{mathcal{L}^*_d}
\left\|
(\mathcal L^*)^p[\mathcal A](\boldsymbol\xi)
\right\|
\le
\sum_{|\boldsymbol\alpha|\le p}
D_{p,\boldsymbol\alpha}^*(\boldsymbol\xi)
\left\|
\partial_{\boldsymbol\xi}^{\boldsymbol\alpha}
\mathcal A(\boldsymbol\xi)
\right\|.
\end{equation}
Here
\(
\boldsymbol\alpha=(\alpha_1,\dots,\alpha_d)\in\mathbb N_0^d,
\,
|\boldsymbol\alpha|
=
\alpha_1+\cdots+\alpha_d.
\)

Moreover, each \(D_{p,\boldsymbol\alpha}^*\) is a finite sum of terms
of the form
\begin{equation}\label{eq:D_structural_form}
C
\prod_{q=1}^{p}
\left|
\partial_{\boldsymbol\xi}^{\boldsymbol\beta_q}
B_{\ell_q}(\boldsymbol\xi)
\right|,\qquad \ell_q\in\{1,\cdots,d\},
\end{equation}
where \(C>0\),
\(
B_\ell(\boldsymbol\xi)
:=
g_{j,n}(\boldsymbol\xi)
\partial_{\xi_\ell}\Phi_{\bm j}^n(\boldsymbol\xi),
\,
g_{j,n}(\boldsymbol\xi)
:=
\frac{1}{
|\nabla_{\boldsymbol\xi}\Phi_{\bm j}^n(\boldsymbol\xi)|^2},
\)
and the multi-indices in \eqref{eq:D_structural_form} satisfy
\(
\sum_{q=1}^{p}|\boldsymbol\beta_q|
=
p-|\boldsymbol\alpha|.
\)
\end{lemma}

\begin{proof}
Set
\(
\bm B
=
(B_1,\dots,B_d)
=
g_{j,n}\nabla_{\boldsymbol\xi}\Phi_{\bm j}^n
\)
and define
\(
\mathscr T f
:=
\nabla_{\boldsymbol\xi}\cdot(\bm Bf)
=
\sum_{\ell=1}^d
\partial_{\xi_\ell}(B_\ell f).
\)
Since
\(
1=-(\sqrt{-1})^2,
\)
the definition of \(\mathcal L^*\) gives
\(\mathcal L^*f
=
\sqrt{-1}\,\mathscr T f.
\)
Consequently,
\(
(\mathcal L^*)^p
=
(\sqrt{-1})^p\mathscr T^p,
\)
and the unimodular factor \((\sqrt{-1})^p\) does not affect the
norm. It therefore suffices to study \(\mathscr T^p\).
We claim that, for every \(p\ge0\),
\begin{equation}\label{eq:Tp_expansion}
\mathscr T^p
=
\sum_{|\boldsymbol\alpha|\le p}
c_{p,\boldsymbol\alpha}(\boldsymbol\xi)
\partial_{\boldsymbol\xi}^{\boldsymbol\alpha},
\end{equation}
where \(c_{p,\boldsymbol\alpha}\) are smooth
functions. For \(p=0\), we take
\(
c_{0,\bm0}=1,
\,
c_{0,\boldsymbol\alpha}=0
\,\text{for }\boldsymbol\alpha\neq\bm0.
\)

Suppose that \eqref{eq:Tp_expansion} holds for some \(p\ge0\). Then
\begin{align*}
\mathscr T^{p+1}u
=
\sum_{\ell=1}^d
\partial_{\xi_\ell}
\left(
B_\ell
\sum_{|\boldsymbol\alpha|\le p}
c_{p,\boldsymbol\alpha}
\partial_{\boldsymbol\xi}^{\boldsymbol\alpha}u
\right)=
\sum_{\ell=1}^d
\sum_{|\boldsymbol\alpha|\le p}\!
\left[
\partial_{\xi_\ell}
(B_\ell c_{p,\boldsymbol\alpha})
\partial_{\boldsymbol\xi}^{\boldsymbol\alpha}u
+
B_\ell c_{p,\boldsymbol\alpha}
\partial_{\boldsymbol\xi}^{\boldsymbol\alpha+\bm e_\ell}u
\right],
\end{align*}
where \(\bm e_\ell\) is the \(\ell\)-th coordinate vector. Hence
\eqref{eq:Tp_expansion} also holds for \(p+1\), with the recurrence
\begin{equation}\label{eq:coefficient_recurrence}
c_{p+1,\boldsymbol\alpha}
=
\sum_{\ell=1}^d
\left[
\partial_{\xi_\ell}
\bigl(
B_\ell c_{p,\boldsymbol\alpha}
\bigr)
+
B_\ell
c_{p,\boldsymbol\alpha-\bm e_\ell}
\right],
\end{equation}
where
\(
c_{p,\boldsymbol\beta}:=0
\)
whenever one component of \(\boldsymbol\beta\) is negative or
\(|\boldsymbol\beta|>p\).

We now prove by induction on \(p\) that each
\(c_{p,\boldsymbol\alpha}\) is a finite linear combination of terms
\(
\prod_{q=1}^{p}
\partial_{\boldsymbol\xi}^{\boldsymbol\beta_q}B_{\ell_q},
\,
\sum_{q=1}^{p}|\boldsymbol\beta_q|
=
p-|\boldsymbol\alpha|.
\) For \(p=1\),
\(
\mathscr T
=
\sum_{\ell=1}^d
\left[
(\partial_{\xi_\ell}B_\ell)
+
B_\ell\partial_{\xi_\ell}
\right],
\)
so the assertion is immediate.

Assume that the assertion holds at level \(p\). In the first term of
\eqref{eq:coefficient_recurrence}, the additional derivative either
falls on the new factor \(B_\ell\) or on one of the \(p\) existing
factors. In either case, the number of \(B\)-factors becomes
\(p+1\), and the total derivative order on these factors increases
by one.
In the second term of \eqref{eq:coefficient_recurrence}, a new
undifferentiated factor \(B_\ell\) is introduced, while
\(
\sum_{q=1}^{p}|\boldsymbol\beta_q|
=
p-|\boldsymbol\alpha-\bm e_\ell|
=
(p+1)-|\boldsymbol\alpha|.
\)
Thus, in every case, \(c_{p+1,\boldsymbol\alpha}\) is a finite linear
combination of products of \(p+1\) differentiated \(B\)-factors
whose total derivative order is
\((p+1)-|\boldsymbol\alpha|\).

Define \(D_{p,\boldsymbol\alpha}^*\) by fully expanding
\(c_{p,\boldsymbol\alpha}\), taking the absolute value of every
numerical coefficient, and replacing every factor
\(
\partial_{\boldsymbol\xi}^{\boldsymbol\beta}
B_\ell
\)
by
\(
\left|
\partial_{\boldsymbol\xi}^{\boldsymbol\beta}
B_\ell
\right|.
\)
Then
\(
|c_{p,\boldsymbol\alpha}(\boldsymbol\xi)|
\le
D_{p,\boldsymbol\alpha}^*(\boldsymbol\xi).
\)

Finally, applying \eqref{eq:Tp_expansion} to the vector-valued
amplitude \(\mathcal A\), componentwise, gives
\begin{align*}
\left\|
(\mathcal L^*)^p[\mathcal A]
\right\|=
\left\|
\mathscr T^p[\mathcal A]
\right\|\le
\sum_{|\boldsymbol\alpha|\le p}
|c_{p,\boldsymbol\alpha}|
\left\|
\partial_{\boldsymbol\xi}^{\boldsymbol\alpha}
\mathcal A
\right\|\le
\sum_{|\boldsymbol\alpha|\le p}
D_{p,\boldsymbol\alpha}^*
\left\|
\partial_{\boldsymbol\xi}^{\boldsymbol\alpha}
\mathcal A
\right\|.
\end{align*}
This proves \eqref{mathcal{L}^*_d}.
\end{proof}

\begin{lemma}
\label{nonstationary_bound3}
Let \(T>0\) be fixed. There exists \(h_0>0\) such that,
for every \(0<h\le h_0\), every integer \(p\ge1\), every
\(t_n\in[0,T]\), and every \(\bm j\in\mathcal J_\Omega^+\), the solution satisfies
\begin{equation}\label{eq:nonstationary_final_d}
   \|\bm u_{\bm j}^n\|
\le
C_d\,\rho^d h^{1+\frac{\gamma d}{2}}
\widetilde M_d (p!)^s
\left(
\frac{\beta_d}
{\rho\,|\bm x_{\bm j}-\bm x_c|}
\right)^p,
\end{equation}
where \(\rho=h^{-\gamma}\). The constants
\(\widetilde M_d>0\) and \(\beta_d>0\) are independent of
\(h,p,n\), and \(\bm j\).
\end{lemma}
\begin{proof}

For brevity, let
\(
\bm m=(ph,\dots,ph)
\)
and set
\(
\bm y_{\bm j}:=\bm x_{\bm j}-\bm x_c.
\)
We define
\begin{equation}\label{eq:phase_shifted_d}
\Phi_{\bm j}^n(\boldsymbol\xi)
:=
\bm y_{\bm j}\cdot\boldsymbol\xi
-
\omega_{\bm m}(\boldsymbol\xi)t_n,\quad\mathcal A(\boldsymbol\xi)
:=
h^{\frac{(1+\gamma)d}{2}}
\mathcal S(\boldsymbol\xi)
\chi_\rho(\boldsymbol\xi),
\quad
\mathcal S(\boldsymbol\xi)
:=
r(\boldsymbol\xi)\bm v_{\bm m}(\boldsymbol\xi).
\end{equation}
Then by Lemma~\ref{nonstationary_bound}, we have
\begin{equation}\label{eq:ibp_representation_d}
\bm u_{\bm j}^n
=
\frac{1}{(2\pi)^d}
\int_{\Omega_{\boldsymbol\xi_c,\rho}}
e^{\sqrt{-1}\Phi_{\bm j}^n(\boldsymbol\xi)}
(\mathcal L^*)^p[\mathcal A](\boldsymbol\xi)
\,\mathrm d\boldsymbol\xi .
\end{equation}

Define \( \bm g_{\bm j,n}(\boldsymbol\xi)
:=
\frac{
\nabla_{\boldsymbol\xi}\Phi_{\bm j}^n(\boldsymbol\xi)
}{
|\nabla_{\boldsymbol\xi}\Phi_{\bm j}^n(\boldsymbol\xi)|^2
}.\)
Since
\(
\nabla_{\boldsymbol\xi}\Phi_{\bm j}^n(\boldsymbol\xi)
=
\bm y_{\bm j}
-
t_n\bm V_g^{\bm m}(\boldsymbol\xi),
\)
we first establish a lower bound for this gradient.
By construction, \(|\bm y_{\bm j}|\ge d_c\) for
\(\bm j\in\mathcal J_\Omega^+\).
By the high-frequency trapping estimate,
\(
\varepsilon_h
:=
\sup_{\boldsymbol\xi\in\Omega_{\boldsymbol\xi_c,\rho}}
|\bm V_g^{\bm m}(\boldsymbol\xi)|
\longrightarrow0
\,\text{as }h\to0.
\)
Hence, for sufficiently small \(h\),
\(
2T\varepsilon_h\le d_c.
\)
Therefore, for \(t_n\in[0,T]\),
\begin{align}
\left|
\nabla_{\boldsymbol\xi}\Phi_{\bm j}^n(\boldsymbol\xi)
\right|
=
\left|
\bm y_{\bm j}
-
t_n\bm V_g^{\bm m}(\boldsymbol\xi)
\right|
\nonumber\ge
|\bm y_{\bm j}|
-
T\varepsilon_h
\nonumber\ge
|\bm y_{\bm j}|-\frac{d_c}{2}
\ge
\frac12|\bm y_{\bm j}|.
\label{eq:phase_gradient_lower_d}
\end{align}

Introduce the scaled frequency
\(
\boldsymbol\eta=h\boldsymbol\xi.
\)
The spatial eigenvalue has the scaling
\(
\sigma_{\bm m}(\boldsymbol\xi)
=
h^{-2}\widetilde\sigma_{\bm m}(\boldsymbol\eta),
\)
and, since \(\Delta t=\lambda h\),
\(
\omega_{\bm m}(\boldsymbol\xi)
=
h^{-1}\widetilde\omega_{\bm m}(\boldsymbol\eta),
\)
where
\(
\widetilde\omega_{\bm m}(\boldsymbol\eta)
=
\frac{2}{\lambda}
\arcsin
\left(
\frac{\lambda}{2}
\sqrt{\widetilde\sigma_{\bm m}(\boldsymbol\eta)}
\right).
\)
Consequently,
\(
\bm V_g^{\bm m}(\boldsymbol\xi)
=
\nabla_{\boldsymbol\xi}\omega_{\bm m}(\boldsymbol\xi)
=
\nabla_{\boldsymbol\eta}
\widetilde\omega_{\bm m}(\boldsymbol\eta).
\)

Let
\(
\boldsymbol\eta_*=(\pi,\dots,\pi).
\)
Under the strict stability assumption,
\(\widetilde\omega_{\bm m}\) is analytic in a fixed complex
polydisc centered at \(\boldsymbol\eta_*\). So there exists \(r_0>0\), independent of \(h\), such that
\(\widetilde\omega_{\bm m}\) is holomorphic in
\(
\mathcal P_{2r_0}
:=
\left\{
\boldsymbol z\in\mathbb C^d:
|z_\ell-\pi|<2r_0,\quad \ell=1,\dots,d
\right\}.
\)
Since the packet is centered near
\(\boldsymbol\eta_*=(\pi,\dots,\pi)\) and
\(
h\rho=h^{1-\gamma}\longrightarrow0,
\)
we have, for sufficiently small \(h\),
\(
h\Omega_{\boldsymbol\xi_c,\rho}
\subset \mathcal P_{r_0}.
\)

Define
\(
\bm q_{\bm j,n}(\boldsymbol z)
:=
\bm y_{\bm j}
-
t_n\nabla_{\boldsymbol z}
\widetilde\omega_{\bm m}(\boldsymbol z).
\)
The holomorphic extension of \(\bm g_{\bm j,n}\) is
\[
\widetilde{\bm g}_{\bm j,n}(\boldsymbol z)
:=
\frac{
\bm q_{\bm j,n}(\boldsymbol z)
}{
\bm q_{\bm j,n}(\boldsymbol z)
\cdot
\bm q_{\bm j,n}(\boldsymbol z)
},
\]
where the dot product is the complex bilinear product, without
complex conjugation.

Since \(\nabla_{\boldsymbol\eta}\widetilde{\omega}_{\boldsymbol m}(\boldsymbol\eta_*)=\mathbf 0\) and \(\widetilde\omega_{\boldsymbol m}\) is holomorphic near \(\boldsymbol\eta_*\), for the fixed \(T>0\) there exists
\(r_0>0\) such that
\( \sup_{\boldsymbol z\in\mathcal P_{2r_0}}
\left|
\nabla_{\boldsymbol z}
\widetilde\omega_{\bm m}(\boldsymbol z)
\right|
\le
\frac{d_c}{4T}.\)
Consequently, for every \(t_n\in[0,T]\) and every \(\bm j\in\mathcal J_\Omega^+\),
\(
t_n
\left|
\nabla_{\boldsymbol z}
\widetilde\omega_{\bm m}(\boldsymbol z)
\right|
\le
\frac{d_c}{4}
\le
\frac14|\bm y_{\bm j}|,
\,
\boldsymbol z\in\mathcal P_{2r_0}.
\)
It follows that
\(
4|\bm q_{\bm j,n}(\boldsymbol z)|
\le
5|\bm y_{\bm j}|
\)
and
\begin{align*}
\left|
\bm q_{\bm j,n}(\boldsymbol z)
\cdot
\bm q_{\bm j,n}(\boldsymbol z)
\right|
&\ge
|\bm y_{\bm j}|^2
-
2|\bm y_{\bm j}|
\left|
t_n\nabla_{\boldsymbol z}\widetilde\omega_{\bm m}
\right|
-
\left|
t_n\nabla_{\boldsymbol z}\widetilde\omega_{\bm m}
\right|^2\ge
\frac7{16}|\bm y_{\bm j}|^2.
\end{align*}
Hence
\(\sup_{\boldsymbol z\in\mathcal P_{2r_0}}
\left|
\widetilde{\bm g}_{\bm j,n}(\boldsymbol z)
\right|
\le
\frac{C_g}{|\bm y_{\bm j}|}.\)

Since
\(
\bm g_{\bm j,n}(\boldsymbol\xi)
=
\widetilde{\bm g}_{\bm j,n}(h\boldsymbol\xi),
\)
the chain rule gives
\(
\partial_{\boldsymbol\xi}^{\boldsymbol\alpha}
\bm g_{\bm j,n}(\boldsymbol\xi)
=
h^{|\boldsymbol\alpha|}
\partial_{\boldsymbol\eta}^{\boldsymbol\alpha}
\widetilde{\bm g}_{\bm j,n}(\boldsymbol\eta).
\)
Applying the multivariate Cauchy estimate in the smaller polydisc
\(\mathcal P_{r_0}\), we obtain
\begin{equation}\label{eq:g_derivative_bound_d_detailed}
\left|
\partial_{\boldsymbol\xi}^{\boldsymbol\alpha}
\bm g_{\bm j,n}(\boldsymbol\xi)
\right|
\le
\frac{C_g}{|\bm y_{\bm j}|}
\boldsymbol\alpha!
\left(
\frac{h}{r_0}
\right)^{|\boldsymbol\alpha|}\qquad\text{with}\qquad\boldsymbol\alpha!
=
\alpha_1!\cdots\alpha_d!.
\end{equation}

Ignoring the unimodular factor,
write
\(
\mathscr T f
:=
\nabla_{\boldsymbol\xi}\cdot
\bigl(
\bm g_{\bm j,n}f
\bigr)
=
\sum_{\ell=1}^d
\partial_{\xi_\ell}
\bigl(
g_{\bm j,n,\ell}f
\bigr).
\)
Set
\(
\mathscr T_\ell f
:=
\partial_{\xi_\ell}
\bigl(
g_{\bm j,n,\ell}f
\bigr).
\)
Then
\(
\mathscr T^p
=
\sum_{\ell_1,\dots,\ell_p=1}^d
\mathscr T_{\ell_p}
\cdots
\mathscr T_{\ell_1}.
\)
There are \(d^p\) coordinate words
\(
(\ell_1,\dots,\ell_p)\in\{1,\dots,d\}^p.
\)
For each fixed coordinate word, repeated use of the Leibniz rule
produces terms of the form
\(
C
\left(
\prod_{q=1}^p
\partial_{\boldsymbol\xi}^{\boldsymbol\alpha_q}
g_{\bm j,n,\ell_q}
\right)
\partial_{\boldsymbol\xi}^{\boldsymbol b}f,
\)
where
\(
\sum_{q=1}^p|\boldsymbol\alpha_q|
=
p-|\boldsymbol b|.
\) Thus, a term containing
\(\partial_{\boldsymbol\xi}^{\boldsymbol b}f\) has 
\(p\) factors involving \(\bm g_{\bm j,n}\), and 
\(p-|\boldsymbol b|\) derivatives distributed among these factors.

To count these terms, define the scalar majorant
\(
M(z)
:=
M_0(1-az)^{-1},
\,
M_0:=\frac{C_g}{|\bm y_{\bm j}|},
\,
a:=\frac{h}{r_0}.
\)
Then
\(
M^{(q)}(0)
=
M_0q!a^q,
\)
which majorizes every derivative in
\eqref{eq:g_derivative_bound_d_detailed}, up to a constant depending
only on \(d\).
With the $p$ coefficient factors labeled, the nested Leibniz
expansion for each coordinate word contains only a subset of the
derivative assignments in the full product derivative.
Replacing coefficient derivatives of order $q$ by the nonnegative
majorants $M^{(q)}(0)$ therefore yields the following coefficientwise
majorant, grouped by the total derivative order on $f$
\[
\frac{\mathrm d^p}{\mathrm dz^p}
\bigl(M(z)^p u(z)\bigr)
=
\sum_{m=0}^p
\binom{p}{m}
\frac{\mathrm d^{p-m}}{\mathrm dz^{p-m}}
M(z)^p\,
u^{(m)}(z).
\]
Since
\(
M(z)^p
=
M_0^p(1-az)^{-p},
\) we have
\begin{align*}
\left.
\frac{\mathrm d^{p-m}}{\mathrm dz^{p-m}}
M(z)^p
\right|_{z=0}
&=
M_0^p
a^{p-m}
p(p+1)\cdots(2p-m-1)=
M_0^p
a^{p-m}
\frac{(2p-m-1)!}{(p-1)!}.
\end{align*}
Therefore, summing over the \(d^p\) coordinate words, for every
multi-index \(\boldsymbol b\) with
\(
|\boldsymbol b|=m,
\)
the structural majorant satisfies
\begin{equation}\label{eq:D_bound_d_detailed}
D_{p,\boldsymbol b}^*(\boldsymbol\xi)
\le
C_d^p
\binom{p}{m}
\frac{(2p-m-1)!}{(p-1)!}
\left(\frac{h}{r_0}\right)^{p-m}
\left(
\frac{C_g}{|\bm y_{\bm j}|}
\right)^p.
\end{equation}

We next estimate the derivatives of
\(
\mathcal A(\boldsymbol\xi)
=
h^{\frac{(1+\gamma)d}{2}}
\mathcal S(\boldsymbol\xi)
\chi_\rho(\boldsymbol\xi).
\) 
Recall that
\(
\bm M^h=h^d\widetilde{\bm M},\,
\sigma_{\bm m}(\boldsymbol\xi)
=h^{-2}\widetilde\sigma_{\bm m}(h\boldsymbol\xi).
\)
Choose an \(h\)-independent analytic eigenvector
\(\widetilde{\bm v}_{\bm m}\) near
\(\boldsymbol\eta_*\), normalized by
\(
\langle
\widetilde{\bm M}\widetilde{\bm v}_{\bm m}(\boldsymbol\eta),
\widetilde{\bm v}_{\bm m}(\boldsymbol\eta)
\rangle=1,
\)
and set
\(
\bm v_{\bm m}(\boldsymbol\xi)
=h^{-\frac{d}{2}}\widetilde{\bm v}_{\bm m}(h\boldsymbol\xi).
\)
Then
\(
r(\boldsymbol\xi)
=\sigma_{\bm m}(\boldsymbol\xi)^{-\frac{1}{2}}
=h\,\widetilde\sigma_{\bm m}(h\boldsymbol\xi)^{-\frac{1}{2}}.
\)
Consequently,
\(
\mathcal S(\boldsymbol\xi)
=
r(\boldsymbol\xi)\bm v_{\bm m}(\boldsymbol\xi)
=
h^{1-\frac{d}{2}}
\widetilde{\mathcal S}(h\boldsymbol\xi),
\)
where
\(
\widetilde{\mathcal S}(\boldsymbol\eta)
:=
\widetilde\sigma_{\bm m}(\boldsymbol\eta)^{-1/2}
\widetilde{\bm v}_{\bm m}(\boldsymbol\eta).
\)

Since \(\widetilde{\mathcal S}\) is analytic in a fixed polydisc of
radius \(R>0\), the multivariate Cauchy estimate gives
\begin{align}
\left\|
\partial_{\boldsymbol\xi}^{\boldsymbol\alpha}
\mathcal S(\boldsymbol\xi)
\right\|
&=
h^{1-\frac d2+|\boldsymbol\alpha|}
\left\|
\partial_{\boldsymbol\eta}^{\boldsymbol\alpha}
\widetilde{\mathcal S}(\boldsymbol\eta)
\right\|
\le
C_S h^{1-\frac d2}
\boldsymbol\alpha!
\left(\frac hR\right)^{|\boldsymbol\alpha|}.
\label{eq:S_derivative_bound_d}
\end{align}

Since
\(
\chi_\rho(\boldsymbol\xi)
=
\prod_{\ell=1}^d
\psi\left(
\frac{\xi_\ell-\xi_{c\ell}}{\rho}
\right),
\)
we have
\begin{equation}\label{eq:chi_derivative_detailed}
    |\partial_{\boldsymbol\xi}^{\boldsymbol\beta}
\chi_\rho(\boldsymbol\xi)|
=
\rho^{-|\boldsymbol\beta|}
\prod_{\ell=1}^d\left|
\psi^{(\beta_\ell)}
\left(
\frac{\xi_\ell-\xi_{c\ell}}{\rho}
\right)\right|\le
C_\chi
(\boldsymbol\beta!)^s
\left(
\frac{B_\chi}{\rho}
\right)^{|\boldsymbol\beta|},
\end{equation}
where the last inequality follows from the Gevrey regularity estimates
\(
|\psi^{(q)}(x)|
\le\!
C_\psi B_\psi^q(q!)^s,
\) and the constants \(C_\chi,B_\chi>0\) depend only on
\(s,d\), and \(\psi\).

For a multi-index \(\boldsymbol b\), the multivariate Leibniz rule
gives
\begin{align}
\partial_{\boldsymbol\xi}^{\boldsymbol b}
\mathcal A
&=
h^{\frac{(1+\gamma)d}{2}}
\sum_{\boldsymbol\alpha\le\boldsymbol b}
\binom{\boldsymbol b}{\boldsymbol\alpha}
\left(
\partial_{\boldsymbol\xi}^{\boldsymbol\alpha}
\mathcal S
\right)
\left(
\partial_{\boldsymbol\xi}^{\boldsymbol b-\boldsymbol\alpha}
\chi_\rho
\right),
\label{eq:A_Leibniz_d}
\end{align}
where
\(
\binom{\boldsymbol b}{\boldsymbol\alpha}
=
\prod_{\ell=1}^d
\binom{b_\ell}{\alpha_\ell}.
\) Substituting
\eqref{eq:S_derivative_bound_d} and
\eqref{eq:chi_derivative_detailed}, we obtain
\begin{align*}
\left\|
\partial_{\boldsymbol\xi}^{\boldsymbol b}
\mathcal A
\right\|
&\le
C_SC_\chi
h^{1+\frac{\gamma d}{2}}
\sum_{\boldsymbol\alpha\le\boldsymbol b}
\binom{\boldsymbol b}{\boldsymbol\alpha}
\boldsymbol\alpha!
\left(\frac{h}{R}\right)^{|\boldsymbol\alpha|}
\bigl((\boldsymbol b-\boldsymbol\alpha)!\bigr)^s
\left(
\frac{B_\chi}{\rho}
\right)^{|\boldsymbol b-\boldsymbol\alpha|}.
\end{align*}
Using
\(
\binom{\boldsymbol b}{\boldsymbol\alpha}
\boldsymbol\alpha!
\bigl((\boldsymbol b-\boldsymbol\alpha)!\bigr)^s
=
\boldsymbol b!\,
\bigl((\boldsymbol b-\boldsymbol\alpha)!\bigr)^{s-1}
\le
(\boldsymbol b!)^s,
\)
we find
\begin{align*}
\left\|
\partial_{\boldsymbol\xi}^{\boldsymbol b}
\mathcal A
\right\|
&\le
C_SC_\chi
h^{1+\frac{\gamma d}{2}}
(\boldsymbol b!)^s
\sum_{\boldsymbol\alpha\le\boldsymbol b}
\left(\frac hR\right)^{|\boldsymbol\alpha|}
\left(
\frac{B_\chi}{\rho}
\right)^{|\boldsymbol b-\boldsymbol\alpha|}.
\end{align*}

Since
\(
h\rho=h^{1-\gamma}\longrightarrow0,
\)
we may assume
\(
h\rho\le R.
\)
The number of multi-indices satisfying
\(\boldsymbol\alpha\le\boldsymbol b\) is
\(
\prod_{\ell=1}^d(b_\ell+1)
\le2^{|\boldsymbol b|}.
\)
Therefore,
\begin{equation}\label{eq:A_derivative_bound_d_detailed}
\left\|
\partial_{\boldsymbol\xi}^{\boldsymbol b}
\mathcal A
\right\|
\le\!
C_SC_\chi
h^{1+\frac{\gamma d}{2}}
(\boldsymbol b!)^s
\!\sum_{\boldsymbol\alpha\le\boldsymbol b}
\frac{1}{\rho^{|\boldsymbol b|}}\!
\left(
\frac{h\rho}{R}
\right)^{|\boldsymbol\alpha|}
\!\!B_\chi^{|\boldsymbol b-\boldsymbol\alpha|}\le
C_A
h^{1+\frac{\gamma d}{2}}
(\boldsymbol b!)^s\!
\left(
\frac{C_A}{\rho}
\right)^{|\boldsymbol b|}.
\end{equation}

By Lemma~\ref{nonstationary_bound2},
\(
\left\|
(\mathcal L^*)^p[\mathcal A]
\right\|
\le
\sum_{|\boldsymbol b|\le p}
D_{p,\boldsymbol b}^*
\left\|
\partial_{\boldsymbol\xi}^{\boldsymbol b}
\mathcal A
\right\|.
\)
Grouping the multi-indices according to
\(
m=|\boldsymbol b|,
\)
and using
\eqref{eq:D_bound_d_detailed} and
\eqref{eq:A_derivative_bound_d_detailed}, we obtain
\begin{align}\label{Lstar_sum_d_detailed}
&\left\|
(\mathcal L^*)^p[\mathcal A]
\right\|\\
&\le\!
C_Ah^{1+\frac{\gamma d}{2}}
\!\!\sum_{m=0}^p
\!\sum_{|\boldsymbol b|=m}\!
\!C_d^p
\binom{p}{m}
\frac{(2p-m-1)!}{(p-1)!}
\left(\frac{h}{r_0}\right)^{p-m}\!
\!\left(
\frac{C_g}{|\bm y_{\bm j}|}
\right)^p
\!\!(\boldsymbol b!)^s
\left(
\frac{C_A}{\rho}
\right)^m.\nonumber
\end{align}

The number of multi-indices
\(\boldsymbol b\in\mathbb N_0^d\) satisfying
\(|\boldsymbol b|=m\) is
\(
N_d(m)
=
\binom{m+d-1}{d-1}.
\)
Moreover,
\(
\boldsymbol b!
\le m!,
\, |\boldsymbol b|=m.
\)
Hence
\(
\sum_{|\boldsymbol b|=m}
(\boldsymbol b!)^s
\le
N_d(m)(m!)^s.
\)
Substituting this into
\eqref{Lstar_sum_d_detailed} gives
\begin{align}
&\left\|
(\mathcal L^*)^p[\mathcal A]
\right\|\\
&\le\!
C_Ah^{1+\frac{\gamma d}{2}}\!\!
\left(
\frac{C_dC_g}{|\bm y_{\bm j}|}
\right)^p\!\!\!
\sum_{m=0}^p\!
N_d(m)\!
\binom{p}{m}
\frac{(2p-m-1)!}{(p-1)!}
(m!)^s
\left(\frac{h}{r_0}\right)^{p-m}\!
\!\!\left(\frac{C_A}{\rho}\right)^m.\nonumber
\label{eq:Lstar_sum2_d}
\end{align}

We now estimate the factorial factor. Since
\(
\binom{p}{m}
=
\frac{p!}{m!(p-m)!},
\)
we have
\begin{align*}
\binom{p}{m}
\frac{(2p-m-1)!}{(p-1)!}
(m!)^s
&=
p\,
\frac{(2p-m-1)!}{(p-m)!}
(m!)^{s-1}.
\end{align*}
Since \(m\le p\),
\(
(m!)^{s-1}\le(p!)^{s-1}.
\)
Furthermore,
\(
\frac{(2p-m-1)!}{(p-m)!}
\le
\frac{(2p)!}{p!}.
\)
Therefore,
\begin{align*}
\binom{p}{m}
\frac{(2p-m-1)!}{(p-1)!}
(m!)^s
&\le
p\,
\frac{(2p)!}{p!}
(p!)^{s-1}=
p
\binom{2p}{p}
(p!)^s\le
p\,4^p(p!)^s.
\end{align*}

Let
\(
\nu
:=
\max
\left\{
r_0^{-1},C_A
\right\},
\)
and observe that
\(\left(\frac{h}{r_0}\right)^{p-m}
\left(\frac{C_A}{\rho}\right)^m
\le
\left(\frac{\nu}{\rho}\right)^p.\)
Thus
\begin{align*}
\left\|
(\mathcal L^*)^p[\mathcal A]
\right\|
&\le
C_Ah^{1+\frac{\gamma d}{2}}
p\,4^p(p!)^s
\left(
\frac{C_dC_g\nu}
{\rho|\bm y_{\bm j}|}
\right)^p
\sum_{m=0}^p N_d(m).
\end{align*}
Using 
\(
\sum_{m=0}^p
\binom{m+d-1}{d-1}
=
\binom{p+d}{d},
\)
we obtain
\(
\sum_{m=0}^pN_d(m)
\le
C_d(p+1)^d.
\)
Consequently,
\begin{align}
\left\|
(\mathcal L^*)^p[\mathcal A]
\right\|
&\le
C_dC_A
h^{1+\frac{\gamma d}{2}}
p(p+1)^d(p!)^s
\left(
\frac{4C_dC_g\nu}
{\rho|\bm y_{\bm j}|}
\right)^p.
\label{eq:Lstar_before_absorb_d}
\end{align}
For any fixed \(\delta>0\), there exists
\(C_{\delta,d}>0\) such that
\(
p(p+1)^d
\le
C_{\delta,d}(1+\delta)^p,
\, p\ge1.
\)
Define
\(
\beta_d
:=
4(1+\delta)C_dC_g\nu
\)
and
\(
\widetilde M_d
:=
C_dC_AC_{\delta,d}.
\)
Then
\begin{equation}\label{eq:Lstar_final_d}
\left\|
(\mathcal L^*)^p[\mathcal A](\boldsymbol\xi)
\right\|
\le
h^{1+\frac{\gamma d}{2}}
\widetilde M_d(p!)^s
\left(
\frac{\beta_d}
{\rho|\bm x_{\bm j}-\bm x_c|}
\right)^p.
\end{equation}

Finally, from \eqref{eq:ibp_representation_d} and
\eqref{eq:Lstar_final_d},
\begin{align*}
\|\bm u_{\bm j}^n\|
\le
\frac1{(2\pi)^d}
\int_{\Omega_{\boldsymbol\xi_c,\rho}}
\left\|
(\mathcal L^*)^p[\mathcal A](\boldsymbol\xi)
\right\|
\,\mathrm d\boldsymbol\xi&\le
\frac{\rho^d}{\pi^d}
h^{1+\frac{\gamma d}{2}}
\widetilde M_d(p!)^s
\left(
\frac{\beta_d}
{\rho|\bm x_{\bm j}-\bm x_c|}
\right)^p.
\end{align*}
Absorbing \(\pi^{-d}\) into the constant \(C_d\) proves
\eqref{eq:nonstationary_final_d}.
\end{proof}

\begin{lemma}
\label{lem:obs_decay}
Fix \(T>0\) and \(\varepsilon\in(0,1)\), and choose
\(1<s<(1-\varepsilon)^{-1}\) and \(\gamma=s(1-\varepsilon)\).
For \(N\in\mathbb N^+\) with \(N\Delta t\le T\)
and sufficiently small \(h\), the observed energy satisfies
\begin{equation*}\label{exponentially decaying energy}
\Delta t \sum_{n=0}^{N-1} \mathcal{E}_{h,\Omega}^{n} \leq C_3(k, \lambda, T, \gamma,d) \exp\left(-\frac{C_4(k,T, \lambda, s, \gamma,d)}{h^{1-\varepsilon}}\right),
\end{equation*}
where $C_3(k, \lambda, T,\gamma,d)>0$ and $C_4(k,T, \lambda, s, \gamma,d)>0$ are independent of $h$.
\end{lemma}
\begin{proof}
Let $p
:=
\left\lfloor
\left(
\frac{\rho d_c}{\beta_d\delta}
\right)^{1/s}
\right\rfloor$, where $\delta > 1$. 
Since \(\rho=h^{-\gamma}\), we have
\(p\to+\infty\) as \(h\to0\). In particular, for sufficiently small
\(h\),
\(
p\ge1,
p>\frac d2.
\)
Moreover, \(p^s
\le
\frac{\rho d_c}{\beta_d\delta}.\)
For sufficiently small \(h\), we have
\(p
\ge
\frac12
\left(
\frac{\rho d_c}{\beta_d\delta}
\right)^{1/s}.\)

By Stirling's inequality, there exists \(C_s>0\) such that
\((p!)^s
\le
C_s p^{s/2}
\left(
\frac{p^s}{e^s}
\right)^p,
p\ge1.\)
Substituting this inequality into
\eqref{eq:nonstationary_final_d}, we obtain
\begin{align}
\|\bm u_{\bm j}^n\|
&\le
C_d
\rho^d h^{1+\frac{\gamma d}{2}}
p^{s/2}
\left(
\frac{p^s\beta_d}
{e^s\rho|\bm y_{\bm j}|}
\right)^p\le
C_d
\rho^d h^{1+\frac{\gamma d}{2}}
p^{s/2}
\left(
\frac{d_c}
{e^s\delta|\bm y_{\bm j}|}
\right)^p.
\label{eq:pointwise_after_stirling_obs}
\end{align}
Let
\(
q:=\frac1{e^s\delta}.
\)
Since \(s>1\) and \(\delta>1\), we have \(0<q<1\). Therefore,
\begin{equation}\label{eq:pointwise_q_obs}
\|\bm u_{\bm j}^n\|
\le
C_d
\rho^d h^{1+\frac{\gamma d}{2}}
p^{s/2}q^p
\left(
\frac{d_c}{|\bm y_{\bm j}|}
\right)^p.
\end{equation}
Since exponential decay dominates polynomial growth, there exists
\(C_{s,q}>0\) such that
\(p^{s/2}q^p
\le
C_{s,q}q^{p/2}, p\ge1.\)
It follows from \eqref{eq:pointwise_q_obs} that
\begin{equation}\label{eq:pointwise_qhalf_obs}
\|\bm u_{\bm j}^n\|
\le
C_d
\rho^d h^{1+\frac{\gamma d}{2}}
q^{p/2}
\left(
\frac{d_c}{|\bm y_{\bm j}|}
\right)^p\le
C_d
\rho^d h^{1+\frac{\gamma d}{2}}
\exp\left(-c_{d,T}h^{-\gamma/s}\right)
\left(
\frac{d_c}{|\bm y_{\bm j}|}
\right)^p,
\end{equation}
where
\(c_{d,T}
:=
\frac{|\log q|}{4}
\left(
\frac{d_c}{\beta_d\delta}
\right)^{1/s}>0.\)

For \(m\ge0\), introduce
\(
\mathcal A_m
:=
\left\{
\bm j\in\mathcal J_\Omega^+:
2^m d_c
\le
|\bm y_{\bm j}|
<
2^{m+1}d_c
\right\}.
\)
The number of grid points in \(\mathcal A_m\) satisfies
\(\#\mathcal A_m
\le
C_d h^{-d}(2^{m+1}d_c)^d.\)
Moreover, for \(\bm j\in\mathcal A_m\),
\(
\left(
\frac{d_c}{|\bm y_{\bm j}|}
\right)^{2p}
\le
2^{-2pm}.
\)
Hence
\begin{align}
\sum_{\bm j\in\mathcal J_\Omega^+}
\left(
\frac{d_c}{|\bm y_{\bm j}|}
\right)^{2p}
\le
\sum_{m=0}^{\infty}
\#\mathcal A_m\,2^{-2pm}
\le
C_dh^{-d}d_c^d
\sum_{m=0}^{\infty}
2^{-m(2p-d)}.
\label{eq:lattice_sum_obs}
\end{align}
Since \(p>d/2\), the geometric series converges. Moreover, for all
sufficiently large \(p\), \(\sum_{\bm j\in\mathcal J_\Omega^+}
\left(
\frac{d_c}{|\bm y_{\bm j}|}
\right)^{2p}
\le
C_dh^{-d}.\)
Squaring \eqref{eq:pointwise_qhalf_obs}, summing over
\(\bm j\in\mathcal J_\Omega^+\), we obtain
\begin{align}
\sum_{\bm j\in\mathcal J_\Omega^+}
\|\bm u_{\bm j}^n\|^2
\le
C_d
h^{-d}\rho^{2d}h^{2+\gamma d}
\exp\left(-2c_{d,T}h^{-\gamma/s}\right).
\label{eq:l2_observed_u}
\end{align}

By \eqref{eq:weak_a}, each LDG auxiliary coefficient is
\(h^{-1}\) times a fixed linear combination of solution coefficients
on its stencil. Thus \(\bm M^h=h^d\bm M_0\),
\(\Delta t=\lambda h\), and the bounded stencil overlap give,
for \(0\le n<N\),
\begin{equation}\label{eq:energy_bound_inv}
\mathcal E_{h,\Omega}^{n}
\le C_{k,\lambda,d}h^{d-2}
\sum_{\bm j\in\mathcal J_\Omega^+}
\left(\|\bm u_{\bm j}^{n+1}\|^2
+\|\bm u_{\bm j}^{n}\|^2\right),
\end{equation}
where \(C_{k,\lambda,d}\) is independent of \(h\).

Substituting \eqref{eq:l2_observed_u} into
\eqref{eq:energy_bound_inv}, we obtain
\begin{align}
\mathcal E_{h,\Omega}^{n}
\le
C_{k,\lambda,d}
\rho^{2d}
h^{\gamma d}
\exp\left(-2c_{d,T}h^{-\gamma/s}\right)\le
C_{k,\lambda,d}
h^{-\gamma d}
\exp\left(-2c_{d,T}h^{-\gamma/s}\right).
\end{align}
Since \(N\Delta t\le T\), for sufficiently small \(h\), summing over
\(n=0,\dots,N-1\) yields
\begin{align}
\Delta t
\sum_{n=0}^{N-1}
\mathcal E_{h,\Omega}^{n}
&\le
T\sup_{0\le n\le N-1}
\mathcal E_{h,\Omega}^{n}
\le
C_{k,\lambda,d}T
h^{-\gamma d}
\exp\left(-2c_{d,T}h^{-\gamma/s}\right),
\\&\le
C_3(k,\lambda,d,\gamma,T)
\exp\left(
-\frac{C_4(d,T,s,k,\gamma,\lambda)}{h^{\gamma/s}}
\right),\nonumber
\label{eq:time_sum_before_absorb}
\end{align}
where \(C_3(k,\lambda,d,\gamma,T),C_4(d,T,s,k,\gamma,\lambda)>0\) are independent of \(h\).
Since \(\gamma/s=1-\varepsilon\), the proof is complete.
    \end{proof}
    With the above Lemmas now established, the proof of Theorem~\ref{non_uniform_observability} follows naturally.
\begin{proof}
    Combining Lemma~\ref{energy_lower_bound} and Lemma~\ref{lem:obs_decay} with the observability inequality, we obtain
\[
C_1(s)
\le
\mathcal E_h^0
\le
C_{T,h}\Delta t\sum_{n=0}^{N-1}\mathcal E_{h,\Omega}^n
\le
C_{T,h}C_3(k,\lambda,d,\gamma,T)
\exp\left(
-\frac{C_4(d,T,s,k,\gamma,\lambda)}{h^{1-\varepsilon}}\right).
\]
Therefore,
\(
 C_{T,h} \geq C_{1,3}(k,T,\lambda, \varepsilon, d) \exp\left(\frac{C_4(k,T,\lambda,\varepsilon,d)}{h^{1-\varepsilon}}\right),
\)
which proves Theorem~\ref{non_uniform_observability}.
\end{proof}

\subsection{Proof of Theorem \ref{theorem restore}}\label{proof of 2.4}
Before giving the proof of the main Theorem, we first state two auxiliary Lemmas.
\begin{lemma}\label{lem:discrete_propagation}
Fix \(h_0>0\). Let
\(
    K_{\delta_k}:=[-\eta_k,\eta_k]^d,
    \,
    \vartheta_{\bm{ph}}(\bm\eta)
    :=
    \Delta t\,
    \omega_{\bm{ph}}\!\left(\frac{\bm\eta}{h}\right),
\)
where $\eta_k$ is chosen as in Lemma~\ref{lem:local_positivity}.
Assume that
\(
    \left|
    \nabla_{\bm\eta}\vartheta_{\bm{ph}}(\bm\eta)
    \right|
    \ge
    \lambda v_g(\delta_k),
    \,
    \bm\eta\in K_{\delta_k}\setminus\{\bm0\},
\)
and 
\(
    0\le
    \vartheta_{\bm{ph}}(\bm\eta)
    \le
    \pi-\gamma_k,
    \,
    \bm\eta\in K_{\delta_k},
\)
for some $\gamma_k>0$ independent of $h$.

Let $\bm m(\bm\eta)$ be a uniformly bounded matrix-valued
multiplier on $K_{\delta_k}$ and let
$\bm b\in L^2(K_{\delta_k})$. For either choice of sign, define
\[
    \bm z_{\bm j}^{\,n}
    :=
    \frac{1}{(2\pi)^d}
    \int_{K_{\delta_k}}
    e^{\sqrt{-1}\bm j\cdot\bm\eta}
    e^{\mp\sqrt{-1}n\vartheta_{\bm{ph}}(\bm\eta)}
    \bm m(\bm\eta)\bm b(\bm\eta)
    \,\mathrm d\bm\eta .
\]
Then, for every bounded set $D\subset\mathbb R^d$, there exists
$C_D>0$, independent of $h$, such that
\[
    \Delta t
    \sum_{n\in\mathbb Z}
    h^d
    \sum_{h\bm j\in D}
    |\bm z_{\bm j}^{\,n}|^2
    \le
    \frac{C_D}{v_g(\delta_k)}
    \frac{h^d}{(2\pi)^d}
    \int_{K_{\delta_k}}
    |\bm b(\bm\eta)|^2
    \,\mathrm d\bm\eta\quad\text{for}\quad0<h\le h_0 .
\]
\end{lemma}
\begin{proof}
Set
\(
    \widetilde{\sigma}_{\bm{ph}}(\bm\eta)
    :=
    h^2
    \sigma_{\bm{ph}}\!\left(\frac{\bm\eta}{h}\right).
\)
By the analyticity and evenness of the physical branch, together
with
\(
    \widetilde{\sigma}_{\bm{ph}}(\bm0)=0,
    \,
    \nabla_{\bm\eta}
    \widetilde{\sigma}_{\bm{ph}}(\bm0)=\bm0,
    \,
    \nabla_{\bm\eta}^2
    \widetilde{\sigma}_{\bm{ph}}(\bm0)
    =
    2\bm I_d,
\)
Taylor's theorem yields constants
$r_*>0$, $C_\sigma>0$, and $C_\nabla>0$,
independent of $h$, such that
\[
    \left|
        \widetilde{\sigma}_{\bm{ph}}(\bm\eta)
        -
        |\bm\eta|^2
    \right|
    \le
    C_\sigma |\bm\eta|^4,
    \qquad
    \left|
        \nabla_{\bm\eta}
        \widetilde{\sigma}_{\bm{ph}}(\bm\eta)
        -
        2\bm\eta
    \right|
    \le
    C_\nabla |\bm\eta|^3,
    \qquad
    |\bm\eta|\le r_*.
\]
Choose $r_1\in(0,r_*]$ sufficiently small so that
\(
    C_\sigma r_1^2\le\frac12
\)
and
\(
    \frac32\lambda^2r_1^2\le2.
\)
Then, for every $0<|\bm\eta|\le r_1$,
\(
    \frac12|\bm\eta|^2
    \le
    \widetilde{\sigma}_{\bm{ph}}(\bm\eta)
    \le
    \frac32|\bm\eta|^2.
\)
In particular,
\(
    \sqrt{2
        \widetilde{\sigma}_{\bm{ph}}(\bm\eta)
    }
    \ge
    |\bm\eta|,\,  4-
    \lambda^2
    \widetilde{\sigma}_{\bm{ph}}(\bm\eta)
    \ge2.
\)
Recall that
\(
    \vartheta_{\bm{ph}}(\bm\eta)
    =
    \Delta t\,
    \omega_{\bm{ph}}\!\left(\frac{\bm\eta}{h}\right)
    =
    2\arcsin\left(
        \frac{\lambda}{2}
        \sqrt{
            \widetilde{\sigma}_{\bm{ph}}(\bm\eta)
        }
    \right).
\)
Hence, for $\bm\eta\ne\bm0$,
\[
    \nabla_{\bm\eta}
    \vartheta_{\bm{ph}}(\bm\eta)
    =
    \frac{
        \lambda
        \nabla_{\bm\eta}
        \widetilde{\sigma}_{\bm{ph}}(\bm\eta)
    }{
        \sqrt{
            \widetilde{\sigma}_{\bm{ph}}(\bm\eta)
        }
        \sqrt{
            4-
            \lambda^2
            \widetilde{\sigma}_{\bm{ph}}(\bm\eta)
        }
    }.
\]

We next derive a quantitative estimate comparing
$\nabla_{\bm\eta}\vartheta_{\bm{ph}}(\bm\eta)$ with
$\lambda\bm\eta/|\bm\eta|$. 
\[
\begin{aligned}
    \left|
        \frac{
            \nabla_{\bm\eta}
            \widetilde{\sigma}_{\bm{ph}}(\bm\eta)
        }{
            2\sqrt{
                \widetilde{\sigma}_{\bm{ph}}(\bm\eta)
            }
        }
        -
        \frac{\bm\eta}{|\bm\eta|}
    \right|
    &\!\le\!
    \frac{
        \left|
            \nabla_{\bm\eta}
            \widetilde{\sigma}_{\bm{ph}}(\bm\eta)
            -
            2\bm\eta
        \right|
    }{2\sqrt{
            \widetilde{\sigma}_{\bm{ph}}(\bm\eta)
        }
    }
    \!+\!
    |\bm\eta|\!
    \left|
        \frac1{
            \sqrt{
                \widetilde{\sigma}_{\bm{ph}}(\bm\eta)
            }
        }
        \!-\!
        \frac1{|\bm\eta|}
    \right|
\!\le\!
    C_1|\bm\eta|^2, 0<|\bm\eta|\le r_1,
\end{aligned}
\]
where
\(
    C_1
    :=
    \frac{C_\nabla}{\sqrt2}
    +
    \sqrt2\,C_\sigma.
\) Also,
\[
\begin{aligned}
    \left|
        \frac{2}{
            \sqrt{
                4-
                \lambda^2
                \widetilde{\sigma}_{\bm{ph}}(\bm\eta)
            }
        }
        -
        1
    \right|=
    \frac{
        \lambda^2
        \widetilde{\sigma}_{\bm{ph}}(\bm\eta)
    }{
        \sqrt{
            4-
            \lambda^2
            \widetilde{\sigma}_{\bm{ph}}(\bm\eta)
        }        \left(
            2+
            \sqrt{
                4-
                \lambda^2
                \widetilde{\sigma}_{\bm{ph}}(\bm\eta)
            }
        \right)
    }\le
    C_2|\bm\eta|^2,
\end{aligned}
\]
for some constant $C_2>0$ independent of $h$.
Thus, we obtain
\[
\begin{aligned}
    \left|
        \nabla_{\bm\eta}
        \vartheta_{\bm{ph}}(\bm\eta)
        -
        \lambda
        \frac{\bm\eta}{|\bm\eta|}
    \right|
    &\le
    \lambda
    \frac{2}{
        \sqrt{
            4-
            \lambda^2
            \widetilde{\sigma}_{\bm{ph}}(\bm\eta)
        }
    }\left|
        \frac{
            \nabla_{\bm\eta}
            \widetilde{\sigma}_{\bm{ph}}(\bm\eta)
        }{
            2\sqrt{
                \widetilde{\sigma}_{\bm{ph}}(\bm\eta)
            }
        }
        -
        \frac{\bm\eta}{|\bm\eta|}
    \right|
    \\
    &\qquad
    +
    \lambda
    \left|
        \frac{2}{
            \sqrt{
                4-
                \lambda^2
                \widetilde{\sigma}_{\bm{ph}}(\bm\eta)
            }
        }
        -
        1
    \right|\le
    \lambda
    \left(
        \sqrt2\,C_1+C_2
    \right)
    |\bm\eta|^2.
\end{aligned}
\]
Choose $r_0\in(0,r_1]$ sufficiently small so that
\(
    \lambda
    \left(
        \sqrt2\,C_1+C_2
    \right) r_0^2
    \le
    \frac{\lambda}{4\sqrt d}.
\)
It follows that
\(
    \left|
        \nabla_{\bm\eta}
        \vartheta_{\bm{ph}}(\bm\eta)
        -
        \lambda
        \frac{\bm\eta}{|\bm\eta|}
    \right|
    \le
    \frac{\lambda}{4\sqrt d},
    \,
    0<|\bm\eta|\le r_0.
\)
For $\ell=1,\ldots,d$ and $s\in\{-1,1\}$, introduce
\(
    \Gamma_{\ell,s}
    :=
    \left\{
        \bm\eta:
        0<|\bm\eta|<r_0,
        s\eta_\ell
        \ge
    \frac{|\bm\eta|}{\sqrt{d}}
    \right\}.
\)
Every nonzero vector $\bm\eta\in\mathbb R^d$ has at least one
component satisfying
\(
   \sqrt d |\eta_\ell|
    \ge
    |\bm\eta|.
\) Hence,
\(
    B_{r_0}(0)\setminus\{\bm0\}
    \subset
    \bigcup_{\ell=1}^d
    \bigcup_{s=\pm1}
    \Gamma_{\ell,s}.
\) For every $\bm\eta\in\Gamma_{\ell,s}$,
\[
\begin{aligned}
    s\partial_{\eta_\ell}
    \vartheta_{\bm{ph}}(\bm\eta)
    &\ge
    \lambda
    \frac{s\eta_\ell}{|\bm\eta|}
    -
    \left|
        \partial_{\eta_\ell}
        \vartheta_{\bm{ph}}(\bm\eta)
        -
        \lambda
        \frac{\eta_\ell}{|\bm\eta|}
    \right|
   \ge
    \frac{\lambda}{\sqrt d}
    -
    \frac{\lambda}{4\sqrt d}
=
    \frac{3\lambda}{4\sqrt d}.
\end{aligned}
\]

We next consider frequencies away from the origin. The set
\(
    K_{\delta_k}\setminus B_{r_0/2}(0)
\)
is compact. By the definition of $v_g(\delta_k)$,
\(
    \left|
        \nabla_{\bm\eta}
        \vartheta_{\bm{ph}}(\bm\eta)
    \right|
    =
    \lambda
    \left|
        \bm V_g^{\bm{ph}}
        \left(\frac{\bm\eta}{h}\right)
    \right|
    \ge
    \lambda v_g(\delta_k),
    \,
    \bm\eta\in
    K_{\delta_k}\setminus\{\bm0\}.
\)
Hence, at every point of
$K_{\delta_k}\setminus B_{r_0/2}(0)$,
there exists
$\ell\in\{1,\ldots,d\}$ such that
\(
    \left|
        \partial_{\eta_\ell}
        \vartheta_{\bm{ph}}
    \right|
    \ge\frac{\lambda v_g(\delta_k)}{\sqrt d}.
\) By continuity, for every such point there exist a neighborhood,
an index $\ell$, and a sign $s\in\{-1,1\}$ such that
\(
    s\partial_{\eta_\ell}
    \vartheta_{\bm{ph}}(\bm\eta)
    \ge
    \frac{
        \lambda v_g(\delta_k)
    }{2\sqrt d}
\)
throughout that neighborhood. Compactness yields a finite
subcover of
$K_{\delta_k}\setminus B_{r_0/2}(0)$.

Moreover, the low-frequency estimate implies
\(
    \left|
        \left|
            \bm V_g^{\bm{ph}}
            \left(\frac{\bm\eta}{h}\right)
        \right|
        -1
    \right|
    \le
    (\sqrt2\,C_1+C_2)
    |\bm\eta|^2,
    \) \(
    0<|\bm\eta|\le r_1.
\)
Therefore,
\(
    \lim_{\bm\eta\to\bm0,\ \bm\eta\ne\bm0}
    \left|
        \bm V_g^{\bm{ph}}
        \left(\frac{\bm\eta}{h}\right)
    \right|
    =
    1,
\)
and hence
\(
    0<v_g(\delta_k)\le1.
\)
Consequently,
\(
    \frac{3\lambda}{4\sqrt d}
    \ge
    \frac{
        \lambda v_g(\delta_k)
    }{2\sqrt d}.
\)

Combining the conic covering near the origin with the finite
covering away from the origin, we obtain a finite partition
\(
    1
    =
    \sum_{r=1}^{R_0}
    \chi_r(\bm\eta)
    \,
    \text{for a.e. }
    \bm\eta\in K_{\delta_k},
\)
such that, for every $r$, there exist
$\ell_r\in\{1,\ldots,d\}$ and $s_r\in\{-1,1\}$ for which
\(
    2\sqrt ds_r
    \partial_{\eta_{\ell_r}}
    \vartheta_{\bm{ph}}(\bm\eta)
    \ge
    \lambda v_g(\delta_k),
    \,
    \bm\eta\in\operatorname{supp}\chi_r.
\) The partition may be chosen so that, for each fixed transverse
frequency, every longitudinal slice of
$\operatorname{supp}\chi_r$ is contained in an interval on which
the above derivative bound holds. Hence the corresponding
one-dimensional change of variables is injective.

Set
\(
    \bm b_r(\bm\eta)
    :=
    \chi_r(\bm\eta)
    \bm m(\bm\eta)
    \bm b(\bm\eta),
\)
and denote by $\bm z_r^n$ the corresponding contribution to
$\bm z^n$. Then
\(
    \bm z_{\bm j}^{\,n}
    =
    \sum_{r=1}^{R_0}
    \bm z_{r,\bm j}^{\,n}.
\)
Since $R_0$ is finite, the Cauchy--Schwarz inequality gives
\(
    \left|
        \bm z_{\bm j}^{\,n}
    \right|^2
    \le
    R_0
    \sum_{r=1}^{R_0}
    \left|
        \bm z_{r,\bm j}^{\,n}
    \right|^2.
\)
It is therefore sufficient to estimate each contribution
separately. Fix $r$ and write
\(
    \ell:=\ell_r,
    \,
    \bm\eta=(\eta_\ell,\bm\eta'),
\)
where $\bm\eta'\in\mathbb R^{d-1}$ denotes the transverse
frequency variables. For fixed $j_\ell$ and $\bm\eta'$, define
\[
    \widetilde{\bm z}_{r,j_\ell}^{\,n}(\bm\eta')
    :=
    \frac1{2\pi}
    \int
    e^{\sqrt{-1}j_\ell\eta_\ell}
    e^{\mp\sqrt{-1}n
        \vartheta_{\bm{ph}}(\eta_\ell,\bm\eta')}
    \bm b_r(\eta_\ell,\bm\eta')
    \,\mathrm d\eta_\ell.
\]

On $\operatorname{supp}\chi_r$, introduce
\(
    \tau
    :=
    s_r
    \vartheta_{\bm{ph}}
    (\eta_\ell,\bm\eta').
\)
Since
\(
    s_r
    \partial_{\eta_\ell}
    \vartheta_{\bm{ph}}
    \ge
    \frac{
        \lambda v_g(\delta_k)
    }{2\sqrt d}
    >0,
\)
the map
\(
    \eta_\ell\longmapsto\tau
\)
is strictly monotone on every relevant longitudinal interval.
Moreover,
\(
    \left|
        \frac{\mathrm d\tau}{\mathrm d\eta_\ell}
    \right|
    =
    \left|
        \partial_{\eta_\ell}
        \vartheta_{\bm{ph}}
        (\eta_\ell,\bm\eta')
    \right|.
\)
After the change of variables, we obtain
\[
\begin{aligned}
    \widetilde{\bm z}_{r,j_\ell}^{\,n}(\bm\eta')
    =
    \frac1{2\pi}
    \int
    e^{\mp\sqrt{-1}s_rn\tau}
    \frac{
        e^{\sqrt{-1}j_\ell
            \eta_\ell(\tau,\bm\eta')}
        \bm b_r(
            \eta_\ell(\tau,\bm\eta'),
            \bm\eta')
    }{
        \left|
            \partial_{\eta_\ell}
            \vartheta_{\bm{ph}}
            \bigl(
                \eta_\ell(\tau,\bm\eta'),
                \bm\eta'
            \bigr)
        \right|
    }
    \,\mathrm d\tau.
\end{aligned}
\]

By the strict stability condition, there exists $\gamma_k>0$,
independent of $h$, such that
\(
    0
    \le
    \vartheta_{\bm{ph}}(\bm\eta)
    \le
    \pi-\gamma_k,
    \,
    \bm\eta\in K_{\delta_k}.
\) Hence the range of $\tau$ is contained in an interval of length
strictly smaller than $2\pi$. Extending the transformed integrand
by zero to an interval of length $2\pi$ and applying Parseval's
identity for Fourier series, we obtain
\[
\begin{aligned}
    \sum_{n\in\mathbb Z}
    \left|
        \widetilde{\bm z}_{r,j_\ell}^{\,n}(\bm\eta')
    \right|^2
    \le
    C
    \int
    \frac{
        |\bm b_r(\eta_\ell,\bm\eta')|^2
    }{
        \left|
            \partial_{\eta_\ell}
            \vartheta_{\bm{ph}}
            (\eta_\ell,\bm\eta')
        \right|
    }
    \,\mathrm d\eta_\ell
\le
    \frac{
        C_d
    }{
        \lambda v_g(\delta_k)
    }
    \int
    |\bm b_r(\eta_\ell,\bm\eta')|^2
    \,\mathrm d\eta_\ell.
\end{aligned}
\]
Applying Parseval's identity in the transverse variables yields
\[
    \sum_{n\in\mathbb Z}
    \sum_{\bm j'\in\mathbb Z^{d-1}}
    \left|
        \bm z_{r,(j_\ell,\bm j')}^n
    \right|^2
    \le
    \frac{
        C_d
    }{
        \lambda v_g(\delta_k)
    }
    \frac1{(2\pi)^d}
    \int_{K_{\delta_k}}
    |\bm b_r(\bm\eta)|^2
    \,\mathrm d\bm\eta.
\]

Since $D\subset\mathbb R^d$ is bounded, there exists $L_D>0$
such that
\(
    D\subset[-L_D,L_D]^d.
\)
Consequently,
\(
    \#\left\{
        j_\ell\in\mathbb Z:
        \exists\,\bm j'\in\mathbb Z^{d-1}
        \text{ such that }
        h(j_\ell,\bm j')\in D
    \right\}
    \le
    \frac{2L_D}{h}+1
    \le
    \frac{C_D}{h}.
\)

Recalling that $\Delta t=\lambda h$, we obtain
\[
    \Delta t\,h^d
    \sum_{n\in\mathbb Z}
    \sum_{h\bm j\in D}
    |\bm z_{r,\bm j}^n|^2
 \le
    \frac{
        C_D
    }{
        v_g(\delta_k)
    }
    \frac{h^d}{(2\pi)^d}
    \int_{K_{\delta_k}}
    |\bm b_r(\bm\eta)|^2
    \,\mathrm d\bm\eta,
\] where $C_D$ has been enlarged if necessary.

Finally, since the partition contains only finitely many functions
and $\bm m$ is uniformly bounded on $K_{\delta_k}$,
\[
\begin{aligned}
    \sum_{r=1}^{R_0}
    \int_{K_{\delta_k}}
    |\bm b_r(\bm\eta)|^2
    \,\mathrm d\bm\eta
    &=
    \sum_{r=1}^{R_0}
    \int_{K_{\delta_k}}
    \left|
        \chi_r(\bm\eta)
        \bm m(\bm\eta)
        \bm b(\bm\eta)
    \right|^2
    \,\mathrm d\bm\eta
  \le
    C
    \int_{K_{\delta_k}}
    |\bm b(\bm\eta)|^2
    \,\mathrm d\bm\eta.
\end{aligned}
\] Combining this estimate with
\(
    \left|
        \bm z_{\bm j}^{\,n}
    \right|^2
    \le
    R_0
    \sum_{r=1}^{R_0}
    \left|
        \bm z_{r,\bm j}^{\,n}
    \right|^2,
\) we conclude that
\[
    \Delta t
    \sum_{n\in\mathbb Z}
    h^d
    \sum_{h\bm j\in D}
    |\bm z_{\bm j}^{\,n}|^2
    \le
    \frac{
        C_D
    }{
        v_g(\delta_k)
    }
    \frac{h^d}{(2\pi)^d}
    \int_{K_{\delta_k}}
    |\bm b(\bm\eta)|^2
    \,\mathrm d\bm\eta.
\]
This completes the proof.
\end{proof}

Before stating the next Lemma, we introduce an auxiliary local
positive energy associated with the LDG variables.
By translation invariance of the LDG relation,
\(
\widehat{\bm Q}_\ell(\bm\xi)
=
\mathcal R_\ell^h(\bm\xi)\widehat{\bm U}(\bm\xi),
\)
where $\mathcal R_\ell^h$ is a matrix-valued Fourier multiplier.
On the uniform mesh, the mass matrix has the scaling
\(
    \bm M^h=h^d\bm M_0.
\)
The LDG
elimination identity gives
\(
    \mathcal K^h(\bm\xi)
    =
    \sum_{\ell=1}^d
    \mathcal R_\ell^h(\bm\xi)^H
    \bm M^h
    \mathcal R_\ell^h(\bm\xi).
\)
We therefore define
\(
    \widetilde{\mathcal K}(\bm\eta)
    :=
    \sum_{\ell=1}^d
    \widetilde{\mathcal R}_\ell(\bm\eta)^H
    \bm M_0
    \widetilde{\mathcal R}_\ell(\bm\eta), 
\) with \( \widetilde{\mathcal R}_\ell(\bm\eta)=h\mathcal R_\ell^h(\bm\eta/h)\),
so that
\(
    \mathcal K^h
    \left(\frac{\bm\eta}{h}\right)
    =
    h^{d-2}
    \widetilde{\mathcal K}(\bm\eta).
\)
For a bounded set $D\subset\mathbb R^d$, let
\(
J(D):=\{\bm j\in\mathbb Z^d:\bm x_{\bm j}\in D\}.
\)
We introduce the auxiliary local positive energy
\[
\begin{aligned}
    \mathfrak E_{h,D}^n
    :={}&
    \frac12
    \sum_{\bm j\in J(D)}
    \left(
        \frac{
            \bm u_{\bm j}^{n+1}
            -
            \bm u_{\bm j}^{n}
        }{\Delta t}
    \right)^H
    \bm M^h
    \left(
        \frac{
            \bm u_{\bm j}^{n+1}
            -
            \bm u_{\bm j}^{n}
        }{\Delta t}
    \right)
    \\
    &+
    \frac14
    \sum_{\ell=1}^d
    \sum_{\bm j\in J(D)}
    \left[
        (\bm q_{\ell,\bm j}^{n+1})^H
        \bm M^h
        \bm q_{\ell,\bm j}^{n+1}
        +
        (\bm q_{\ell,\bm j}^{n})^H
        \bm M^h
        \bm q_{\ell,\bm j}^{n}
    \right].
\end{aligned}
\] 
\begin{lemma}\label{lem:local_energy_halfwaves}
Let $D\subset\mathbb R^d$ be bounded, and let
$\{\bm U^n\}_{n\in\mathbb Z}$ be a solution of
\eqref{eq:global_fully_discrete} whose Fourier support is contained
in $I_{\delta_k}^d$ and whose Fourier coefficients belong to the
physical eigenspace.
Then there exist a finite integer $J\ge1$, a constant $C_E>0$
independent of $h$, and half-wave variables
$\bm Z_{\nu,\pm}^n$, $\nu=1,\ldots,J$, such that
\[
    \mathfrak E_{h,D}^n
    \le
    C_E
    \sum_{\nu=1}^{J}
    \sum_{\pm}
    h^d
    \sum_{\bm j\in J(D)}
    \left|
        \bm Z_{\nu,\pm,\bm j}^{\,n}
    \right|^2.
\]
Each half-wave variable has the representation
\[
    \bm Z_{\nu,\pm,\bm j}^{\,n}
    =
    \frac{1}{(2\pi)^d}
    \int_{K_{\delta_k}}
    e^{\sqrt{-1}\bm j\cdot\bm\eta}
    e^{\mp\sqrt{-1}n
        \vartheta_{\bm{ph}}(\bm\eta)}
    \bm m_{\nu,\pm}(\bm\eta)
    b_\pm(\bm\eta)
    \,\mathrm d\bm\eta,
\]
where
\(
    \sup_{\bm\eta\in K_{\delta_k}}
    \left\|
        \bm m_{\nu,\pm}(\bm\eta)
    \right\|
    \le C_E.
\)
Moreover, $b_\pm$ can be normalized so that
\[
    \frac{h^d}{(2\pi)^d}
    \int_{K_{\delta_k}}
    \left(
        |b_+(\bm\eta)|^2
        +
        |b_-(\bm\eta)|^2
    \right)
    \,\mathrm d\bm\eta
    =
    \mathcal E_h^0.
\]
\end{lemma}
\begin{proof}
    Let $\alpha_\pm(\bm\xi)$ denote the two temporal half-wave
coefficients in the original Fourier variable $\bm\xi$. Since the
solution is supported on the physical branch,
\[
\begin{aligned}
    \widehat{\bm U}^{\,n}(\bm\xi)
    =
    \Big[
        \alpha_+(\bm\xi)
        e^{-\sqrt{-1}n
            \Delta t\,\omega_{\bm{ph}}(\bm\xi)}
        +
        \alpha_-(\bm\xi)
        e^{\sqrt{-1}n
            \Delta t\,\omega_{\bm{ph}}(\bm\xi)}
    \Big]
    \bm v_{\bm{ph}}(\bm\xi).
\end{aligned}
\]
Introduce $\bm\eta=h\bm\xi$ and define
\(
\widetilde{\bm v}_{\bm{ph}}(\bm\eta)
=h^{d/2}\bm v_{\bm{ph}}(\bm\eta/h),\,
a_\pm(\bm\eta)=h^{-3d/2}\alpha_\pm(\bm\eta/h).
\)
Then the inverse Fourier transform becomes
\[
\begin{aligned}
    \bm u_{\bm j}^n
    =
    \frac{1}{(2\pi)^d}
    \int_{K_{\delta_k}}
    e^{\sqrt{-1}\bm j\cdot\bm\eta}
    \Big[
        a_+(\bm\eta)
        e^{-\sqrt{-1}n
            \vartheta_{\bm{ph}}(\bm\eta)}
        +
        a_-(\bm\eta)
        e^{\sqrt{-1}n
            \vartheta_{\bm{ph}}(\bm\eta)}
    \Big]
    \widetilde{\bm v}_{\bm{ph}}(\bm\eta)
    \,\mathrm d\bm\eta,
\end{aligned}
\]
where
\(
    \vartheta_{\bm{ph}}(\bm\eta)
    \!=
    \Delta t\,
    \omega_{\bm{ph}}
    \left(\frac{\bm\eta}{h}\right).
\)
Define
\(
    \widetilde{\sigma}_{\bm{ph}}(\bm\eta)
   \! :=
    h^2
    \sigma_{\bm{ph}}
    \left(\frac{\bm\eta}{h}\right).
\) Using
\(
    \mathcal K^h
    \left(\frac{\bm\eta}{h}\right)
    =
    h^{d-2}
    \widetilde{\mathcal K}(\bm\eta),
    \)\\\(
    \bm M^h=h^d\bm M_0,
\)
we obtain
\(
    \widetilde{\mathcal K}(\bm\eta)
    \widetilde{\bm v}_{\bm{ph}}(\bm\eta)
    =
    \widetilde{\sigma}_{\bm{ph}}(\bm\eta)
    \bm M_0
    \widetilde{\bm v}_{\bm{ph}}(\bm\eta).
\)
We normalize the physical eigenvector by
\(
    \widetilde{\bm v}_{\bm{ph}}(\bm\eta)^H
    \bm M_0
    \widetilde{\bm v}_{\bm{ph}}(\bm\eta)
    =
    1.
\)

Fix
$\bm\eta\in K_{\delta_k}\setminus\{\bm0\}$ and write
\(
    \vartheta
    :=
    \vartheta_{\bm{ph}}(\bm\eta),
    \,
    \widetilde{\sigma}
    :=
    \widetilde{\sigma}_{\bm{ph}}(\bm\eta),
    \,
    z:=e^{-\sqrt{-1}\vartheta}.
\)
Since $\Delta t=\lambda h$, the fully discrete dispersion relation
is
\(
    \frac{4}{\lambda^2}
    \sin^2\left(\frac{\vartheta}{2}\right)
    =
    \widetilde{\sigma}.
\)

Consider
\(
    c_n
    :=
    a_+z^n+a_-z^{-n}.
\)
Then
\(
    h\frac{c_{n+1}-c_n}{\Delta t}
    =
    a_+z^n\rho_+
    +
    a_-z^{-n}\rho_-,
\)
where
\(
    \rho_+
    :=
    \frac{z-1}{\lambda},
    \,
    \rho_-
    :=
    \frac{z^{-1}-1}{\lambda}.
\)
The dispersion relation gives
\(
    |\rho_+|^2
    =
    |\rho_-|^2
    =
    \widetilde{\sigma}.
\) The contribution of this frequency to the conserved energy, after
extracting the common factor $h^{d-2}$, is
\(
    \frac12
    \left|
        h\frac{c_{n+1}-c_n}{\Delta t}
    \right|^2
    +
    \frac{\widetilde{\sigma}}2
    \operatorname{Re}
    \left(
        c_{n+1}\overline{c_n}
    \right).
\) The diagonal contribution associated with $a_+$ is
\(
  \widetilde{\sigma}
    \cos^2\left(\frac{\vartheta}{2}\right)
    |a_+|^2,
\)
and the same identity holds for $a_-$.

For the mixed terms, the dispersion relation implies
\(
    \frac{(z-1)^2}{\lambda^2}
    +
    \widetilde{\sigma}z
    =
    0,
\) because
\(
    \frac{(z-1)^2}{z}
  =
    -4\sin^2\left(\frac{\vartheta}{2}\right)
   =
    -\lambda^2\widetilde{\sigma}.
\)
Hence the mixed kinetic and stiffness contributions cancel exactly.

Define
\(
    \widetilde q_{\bm{ph}}(\bm\eta)
    :=
    \widetilde{\sigma}_{\bm{ph}}(\bm\eta)
    \cos^2\left(
        \frac{
            \vartheta_{\bm{ph}}(\bm\eta)
        }2
    \right).
\)
The physical energy density is therefore
\(
    \widetilde q_{\bm{ph}}(\bm\eta)
    \left(
        |a_+(\bm\eta)|^2
        +
        |a_-(\bm\eta)|^2
    \right).
\) By the strict stability condition on the filtered physical branch,
there exists $\gamma_k>0$, independent of $h$, such that
\(
    0
    \le
    \vartheta_{\bm{ph}}(\bm\eta)
    \le
    \pi-\gamma_k,
    \,
    \bm\eta\in K_{\delta_k}.
\) Hence
\(
    \cos^2\left(
        \frac{
            \vartheta_{\bm{ph}}(\bm\eta)
        }2
    \right)
    \ge
    \sin^2\left(\frac{\gamma_k}{2}\right),
\) and therefore
\(
    c_k
    \widetilde{\sigma}_{\bm{ph}}(\bm\eta)
    \le
    \widetilde q_{\bm{ph}}(\bm\eta)
    \le
    \widetilde{\sigma}_{\bm{ph}}(\bm\eta).
\)

Define
\(
    b_\pm(\bm\eta)
    :=
    \frac1h
    \sqrt{
        \widetilde q_{\bm{ph}}(\bm\eta)
    }
    \,a_\pm(\bm\eta).
\) The exact energy computation and Parseval's identity yield
\(
\begin{aligned}
    \mathcal E_h^0
   =
    \frac{h^d}{(2\pi)^d}
    \int_{K_{\delta_k}}
    \left(
        |b_+(\bm\eta)|^2
        +
        |b_-(\bm\eta)|^2
    \right)
    \,\mathrm d\bm\eta.
\end{aligned}
\)

For the temporal difference, define
\(
    \rho_\pm(\bm\eta)
    :=
    \frac{
        e^{\mp\sqrt{-1}
        \vartheta_{\bm{ph}}(\bm\eta)}
        -1
    }{\lambda}.
\)
Then
\(
    |\rho_\pm(\bm\eta)|^2
    =
    \widetilde{\sigma}_{\bm{ph}}(\bm\eta).
\)
The temporal multiplier is
\(
    \bm m_{0,\pm}(\bm\eta)
    :=
    \frac{
        \rho_\pm(\bm\eta)
        \bm M_0^{1/2}
        \widetilde{\bm v}_{\bm{ph}}(\bm\eta)
    }{
        \sqrt{
            \widetilde q_{\bm{ph}}(\bm\eta)
        }
    }.
\)
For $\bm\eta\ne\bm0$,
\[
\begin{aligned}
    \left\|
        \bm m_{0,\pm}(\bm\eta)
    \right\|^2
    =
    \frac{
        \widetilde{\sigma}_{\bm{ph}}(\bm\eta)
    }{
        \widetilde q_{\bm{ph}}(\bm\eta)
    }
    =
    \frac1{
        \cos^2\left(
            \vartheta_{\bm{ph}}(\bm\eta)/2
        \right)
    }
    \le
    \frac1{\sin^2(\gamma_k/2)}.
\end{aligned}
\]

We next consider the LDG auxiliary variables. Since
\(
    \widehat{\bm Q}_\ell(\bm\xi)
    =
    \mathcal R_\ell^h(\bm\xi)
    \widehat{\bm U}(\bm\xi)
\)
and
\(
    \mathcal R_\ell^h
    \left(\frac{\bm\eta}{h}\right)
    =
    \frac1h
    \widetilde{\mathcal R}_\ell(\bm\eta),
\) we have
\[
\begin{aligned}
    \bm q_{\ell,\bm j}^n
    =
    \frac{1}{(2\pi)^d}\!\!
    \int_{K_{\delta_k}}
    \!\!\!\!\!\!e^{\sqrt{-1}\bm j\cdot\bm\eta}
    \frac1h
    \widetilde{\mathcal R}_\ell(\bm\eta)
    \widetilde{\bm v}_{\bm{ph}}(\bm\eta)
    \Big[
        a_+(\bm\eta)
        e^{-\sqrt{-1}n
            \vartheta_{\bm{ph}}(\bm\eta)}
        +
        a_-(\bm\eta)
        e^{\sqrt{-1}n
            \vartheta_{\bm{ph}}(\bm\eta)}
    \Big]
    \mathrm d\bm\eta.
\end{aligned}
\] The factorization of $\widetilde{\mathcal K}$ and the scaled
physical eigenvalue equation give
\[
\begin{aligned}
    \sum_{\ell=1}^d
    \left\|
        \bm M_0^{1/2}
        \widetilde{\mathcal R}_\ell(\bm\eta)
        \widetilde{\bm v}_{\bm{ph}}(\bm\eta)
    \right\|^2
    =
    \widetilde{\bm v}_{\bm{ph}}(\bm\eta)^H
    \widetilde{\mathcal K}(\bm\eta)
    \widetilde{\bm v}_{\bm{ph}}(\bm\eta)
   =
    \widetilde{\sigma}_{\bm{ph}}(\bm\eta).
\end{aligned}
\] Consequently, for every $\ell$,
\(
    \left\|
        \bm M_0^{1/2}
        \widetilde{\mathcal R}_\ell(\bm\eta)
        \widetilde{\bm v}_{\bm{ph}}(\bm\eta)
    \right\|
    \le
    \sqrt{
        \widetilde{\sigma}_{\bm{ph}}(\bm\eta)
    }.
\) Define the spatial multipliers by
\(
    \bm m_{\ell,\pm}(\bm\eta)
    :=
    \frac{
        \bm M_0^{1/2}
        \widetilde{\mathcal R}_\ell(\bm\eta)
        \widetilde{\bm v}_{\bm{ph}}(\bm\eta)
    }{
        \sqrt{
            \widetilde q_{\bm{ph}}(\bm\eta)
        }
    }.
\)
Then, for $\bm\eta\ne\bm0$, \( \left\|
        \bm m_{\ell,\pm}(\bm\eta)
    \right\|^2
    \le
    \frac1{\sin^2(\gamma_k/2)}.\) At $\bm\eta=\bm0$, where both
$\widetilde{\sigma}_{\bm{ph}}$ and
$\widetilde q_{\bm{ph}}$ vanish, the multipliers may be assigned
arbitrary bounded values, since this single point has zero
Lebesgue measure.

The terms at time $n+1$ differ from those at time $n$ only through
the unimodular factors
\(
    e^{\mp\sqrt{-1}
        \vartheta_{\bm{ph}}(\bm\eta)},
\)
and hence satisfy the same multiplier bounds.
It follows that every temporal and auxiliary-variable component
entering $\mathfrak E_{h,D}^n$ is a sum of two half-wave variables
of the form
\[
    \bm Z_{\nu,\pm,\bm j}^{\,n}
    =
    \frac{1}{(2\pi)^d}
    \int_{K_{\delta_k}}
    e^{\sqrt{-1}\bm j\cdot\bm\eta}
    e^{\mp\sqrt{-1}n
        \vartheta_{\bm{ph}}(\bm\eta)}
    \bm m_{\nu,\pm}(\bm\eta)
    b_\pm(\bm\eta)
    \,\mathrm d\bm\eta,
\]
where
\(    \sup_{\bm\eta\in K_{\delta_k}}
    \left\|
        \bm m_{\nu,\pm}(\bm\eta)
    \right\|
    \le C_E
\)
for a constant $C_E$ independent of $h$.

Since $k$ and $d$ are fixed, only finitely many temporal and
auxiliary-variable components occur. Hence $J$ is independent of
$h$. Using
\(
    |x+y|^2
   \!\!  \le\! 
    2(|x|^2\!+\!|y|^2)
\)
for each decomposition into the two half-waves, and summing over
the finitely many components, we have
\(
    \mathfrak E_{h,D}^n
    \! \le\! 
    C_E
    \sum_{\nu=1}^{J}
    \sum_{\pm}
    h^d
    \sum_{\bm j\in J(D)}
   \!  \left|
        \bm Z_{\nu,\pm,\bm j}^{\,n}
    \right|^2.
\)
This completes the proof.
\end{proof}

We now turn to the proof of  Theorem \ref{theorem restore}.
\begin{proof}
Let $Q=[-1,1]^d$ and $0<h\le h_0$, with $h_0>0$ fixed.
By the LDG factorization and the partition of cell centers
between $Q$ and $\Omega$,
\[
\mathcal E_h^n
+\frac14\|\bm U^{n+1}-\bm U^n\|_{\mathbb K^h}^2
=\mathcal E_{h,\Omega}^n+\mathfrak E_{h,Q}^n.
\]
Energy conservation and $N\Delta t=T$ therefore give
\[
T\mathcal E_h^0
\le\Delta t\sum_{n=0}^{N-1}\mathcal E_{h,\Omega}^n
+\Delta t\sum_{n=0}^{N-1}\mathfrak E_{h,Q}^n.
\]
Applying Lemmas~\ref{lem:local_energy_halfwaves}
and~\ref{lem:discrete_propagation} on the fixed set $Q$, we obtain
\[
\begin{aligned}
\Delta t\sum_{n=0}^{N-1}\mathfrak E_{h,Q}^n
&\le C\sum_{\nu=1}^{J}\sum_{\pm}
\Delta t\sum_{n\in\mathbb Z}
h^d\sum_{\bm j\in J(Q)}
|\bm Z_{\nu,\pm,\bm j}^{\,n}|^2\\
&\le\frac{C}{v_g(\delta_k)}
\frac{h^d}{(2\pi)^d}\int_{K_{\delta_k}}
\bigl(|b_+|^2+|b_-|^2\bigr)\,\mathrm d\bm\eta
=\frac{C}{v_g(\delta_k)}\mathcal E_h^0,
\end{aligned}
\]
where $C>0$ is independent of $h$ and $T$.
Thus, for $T>T^*(\delta_k):=C/v_g(\delta_k)$,
\[
\mathcal E_h^0
\le\frac{\Delta t}{T-T^*(\delta_k)}
\sum_{n=0}^{N-1}\mathcal E_{h,\Omega}^n,
\]
and uniform observability follows.
\end{proof}

\section{Conclusions}\label{section 4}

This paper establishes the failure and recovery of uniform observability for fully discrete arbitrary high-order approximations of the wave equation in \(d\) dimensions. The spectral analysis shows that the physical branch reproduces the continuous propagation speed at low frequencies but develops vanishing group velocities at high-frequency critical points. This degeneracy generates trapped wave packets and yields an exponential lower bound of the form \(\exp(Ch^{-(1-\varepsilon)})\) for the observability constant.
Uniform observability is restored for arbitrary fixed $k$ by projecting onto the physical branch and filtering frequencies with degenerate group velocities. Theoretical analysis relates the sufficient observation time to the minimum physical group velocity, while numerical results suggest that higher-order approximations permit wider retained bands at a prescribed positive
group-speed threshold. This offers a potential advantage of higher-order methods in wave-control computation, as less restrictive filtering allows more physical-frequency content to be retained.


\bibliographystyle{siamplain}
\bibliography{refe}
\end{document}